\documentclass[12pt,a4paper]{amsart}
\makeatletter
\renewcommand\normalsize{%
    \@setfontsize\normalsize{11.7}{14pt plus .3pt minus .3pt}%
    \abovedisplayskip 10\p@ \@plus4\p@ \@minus4\p@
    \abovedisplayshortskip 6\p@ \@plus2\p@
    \belowdisplayshortskip 6\p@ \@plus2\p@
    \belowdisplayskip \abovedisplayskip}
\renewcommand\small{%
    \@setfontsize\small{9.5}{12\p@ plus .2\p@ minus .2\p@}%
    \abovedisplayskip 8.5\p@ \@plus4\p@ \@minus1\p@
    \belowdisplayskip \abovedisplayskip
    \abovedisplayshortskip \abovedisplayskip
    \belowdisplayshortskip \abovedisplayskip}
\renewcommand\footnotesize{%
    \@setfontsize\footnotesize{8.5}{9.25\p@ plus .1pt minus .1pt}%%
    \abovedisplayskip 6\p@ \@plus4\p@ \@minus1\p@
    \belowdisplayskip \abovedisplayskip	
    \abovedisplayshortskip \abovedisplayskip
    \belowdisplayshortskip \abovedisplayskip}
\ifdefined\pdfpagewidth
\else
\fi
\calclayout
\makeatother

\usepackage{amssymb,amscd,amsthm,mathrsfs,enumitem,color,transparent,xcolor}
\usepackage[utf8]{inputenc}
\usepackage{epsfig,graphicx,psfrag,mathrsfs,float}  %figuras  \usepackage{wrapfig,subfigure}
\newtheorem*{thm*}{Theorem}
\newtheorem*{prop*}{Proposition}
\newtheorem*{thmA}{Theorem A}
\newtheorem*{thmB}{Theorem B}
\newtheorem*{thmC}{Theorem C}

\newtheorem{thm}{Theorem}[section]

\newtheorem{lemma}[thm]{Lemma}

\newtheorem{prop}[thm]{Proposition}
\newtheorem{cor}[thm]{Corollary}

\newcommand{\bi}{\begin{itemize}}
\newcommand{\ei}{\end{itemize}}

\theoremstyle{definition}

\theoremstyle{remark}
\newtheorem{obs}[thm]{Remark}

\newcommand{\D}{\mathbb{D}}

\newcommand{\T}{\mathbb{T}}
\newcommand{\R}{\mathbb{R}}
\newcommand{\Z}{\mathbb{Z}}
\newcommand{\N}{\mathbb{N}}

\newcommand{\homeo}{\mathrm{Homeo}_0(\A)}

\newcommand{\A}{\mathbb{A}}

\title[Conditions implying Annular Chaos]{Conditions implying annular chaos}

\author{Alejandro Passeggi}
\address[Passeggi]{CMAT, Facultad de Ciencias, Universidad de la Rep\'{u}blica,
Igu\'{a} 4225, 11400 Montevideo, Uruguay} \email{apasseggi@cmat.edu.uy}

\author{F\'{a}bio Armando Tal}
\address[Tal]{Instituto de Matem\' atica e Estat\' istica da Universidade de S\~ao Paulo,
R. do Mat\~ ao, 1010 - Vila Universitaria, S\~ ao Paulo, Brasil}
\email{fabiotal@ime.usp.br}
\thanks{A.P.  has been partially supported by CSIC and ANII. F.T. has been partially supported by FAPESP and CNPq-Brasil}

\begin{document}

\maketitle

\begin{abstract}

This work investigates topological chaos for homeomorphisms of the open annulus, introducing a new set of sufficient conditions based on points with distinct rotation numbers and their topological relations to  invariant continua. These conditions allow us to formulate classic methods for verifying annular chaos in a finitely verifiable version supported on basic properties of the map. The results pave the way for simple computer-assisted proofs of chaos in a wide range of annular maps, including many well--known examples, and we present these proofs
for some analytic families, demonstrating the effectiveness of the method.
%On the theoretical side, one of the consequences of the established conditions permits the proof of a folkloric conjecture about the relation between topological entropy and rotation sets.

\end{abstract}

\section{Introduction}\label{s.intro}

In the study of dynamical systems, the perception that there are various natural phenomena that can be modeled in a simple way and yet behave in a rather unpredictable way dates back to the beginning of the last century. This emergence of ``chaotic'' behavior, which eventually led to the development of the mathematical notion of \emph{chaos}, first became apparent in two qualitatively very different settings: The first one inside Hamiltonian dynamics, as a result of the pioneering work of Poincar\'{e} on \emph{the three-body problem}; the second appeared in dissipative systems associated with periodically forced second-order o.d.e. arising from electric circuits, such as the
\emph{Van der Pol equation}. 

The study of this unpredictability and the development of the concept of chaos was one of the driving forces in the theory of Dynamical Systems in the last century, leading to the establishment of formal notions and also the comprehension of some well-known paradigmatic mathematical examples. From a topological perspective, probably the best notion of chaotic dynamical system is that of a system having positive topological entropy, which is used in this work.

However, despite the wide advancements in the field, a significant challenge remains: determining whether a specific dynamical system exhibits positive topological entropy, particularly in the absence of global topological conditions. This issue plagues many important classical examples and consequently, current theoretical frameworks for chaos struggle to translate directly to proofs  for specific systems. In practice, rigorous proofs of the existence of chaos are often available only for a relatively small set of families, and the parameter ranges where such proofs work is often extremely narrow. Moreover, the classical techniques are hard to be adapted to specific maps, and hence cannot be used in a systematic way. References like \cite{Clone,Cltwo,Levsecondorder,HolmesHoket,holmesmeli,laisangyoung,haiduc} exemplify these limitations. In contrast, in several areas of contemporary science there is a plethora of articles concluding that certain systems are chaotic after numerical simulations, without any rigorous proof. The study of the \emph{double physical pendulum} serves as a notorious example of these situations.

In this work we deal with \emph{annular homeomorphisms}, that is, homeomorphisms of the two-dimensional open annulus. The study of these maps is of well-known relevance, since renowned models with two degrees of freedom have global Poincar\'{e} sections that are annuli. In particular, this is the case for some  restrictions of the $3$-body problem as well as the  above mentioned Van der Pol Equation, and many others in various fields of science. These dynamics can often be displayed as a parameterized family of maps which contain an initial \emph{integrable} model with non-trivial rotation sets, and it was the study of these systems that led Birkhoff, continuing Poincar\'{e}'s work, to implement the systematic study of annular discrete dynamics with non-trivial rotational behavior.
Since then, the theory has been developed in several different branches and here we focus on the topological approach. 

Building on Birkhoff's foundational work, the study of annular dynamics from a topological viewpoint has focused on the relationship between rotation sets and entropy. Two key concepts have emerged, dating back all the way to Birkhoff himself \cite{birk1,birk2}: \emph{instability regions} in the conservative setting and \emph{circloids} with non-trivial rotational behavior in the dissipative setting (generalizing Birkhoff attractors). Instability regions are open invariant annular regions containing  points with distinct rotation numbers and without invariant essential continua (as studied in \cite{birk2,arnaud,mather1,mather2,mather3,frankspatrice}). Circloids, on the other hand, are invariant continua irreducibly separating the annulus into two unbounded components, and circloids carrying distinct rotation numbers have been studied in \cite{birk1,Levsecondorder,herman,lecalveztesis,gagltre,
veer,casdalgi,HolmesHoket,Gukholmes,walker,barge,koropeki,haiduc,laisangyoung}. These concepts are crucial for classifying annular regions that exhibit rotational chaos, that is, those for which a power of the map has a topological horseshoe with nontrivial rotational behavior, as defined for instance in \cite{paposa}, and meaningful positive results were obtained in the last years. Namely,
instability regions imply the existence of rotational horseshoes, as proved in
\cite{frankshandel} for smooth maps with techniques based on Nielsen-Thurston theory, and then
in \cite{forcing} for the $C^0$ case as an outcome of the \emph{forcing technique} based on Brouwer Foliations; and \emph{attracting}
circloids with non-trivial rotation sets imply the existence of rotational horseshoes,
proved first in \cite{paposa} using basic topological techniques and classical results, and then in \cite{lecalveztal} based also on Forcing theory.
Concerning this last result, there is a folkloric conjecture \cite{casdalgi},\cite{koropeki} where the attracting hypothesis was unnecessary, establishing that the presence of a circloid with nontrivial rotational  behavior implies the existence of horseshoes. Despite these developments, we find the same problem mentioned above, as they are still difficult to implement in specific cases: how can one determine if a given map exhibits
an instability region or an invariant circloid with non-trivial rotation set?

\medskip

The purpose of this paper is to address the identified shortcomings by introducing simple, finitary criteria based on basic  properties of the maps to ensure annular chaos. The principal achievement lies in transforming the classical characterization of annular chaos into finitary statements that rely solely on fundamental properties of the maps, without requiring a detailed analysis of the specific map under consideration. This shift is important already at the foundational level of the theory: when working with Poincaré sections, studying (for instance) the differential of the return map is typically a very delicate task. Even in Poincaré’s later works, one can read that developing a topological approach to the dynamical properties of differential equations should be of central importance. 
Let us illustrate the type of results we aim to describe with two particular statements which concern the conservative and dissipative settings respectively. Before doing so, 
we need to introduce our key concept. 

\smallskip

A pair of topological closed disks $U_0$, $U_1$ containing fixed points  $x_0,x_1$ of an annular map $f$
form a \emph{$3$-disjoint pair of neighborhoods for} $x_0, x_1$ ($3$-dpn) whenever 
\[
\bigcup_{j=0}^{3} f^j(U_0) \quad \text{and} \quad \bigcup_{j=0}^{3} f^j(U_1)
\]
remain both inessential (that is, contained in a topological disk) and disjoint. We say that two fixed points $x_0,x_1$  are \emph{3-Birkhoff related} if there exists a $3$-dpn $U_0, U_1$ such that the forward orbit of $U_0$ intersects $U_1$ and the forward orbit of $U_1$ intersects $U_0$.

A homeomorphism of a surface is said to be non-wandering if it has no wandering points.

\begin{thm*}

Let $f$ be a non-wandering annular homeomorphism homotopic to the identity,
and let $x_0,x_1$ be a pair of fixed points having different rotation numbers.
Then, if $x_0,x_1$ are $3$-Birkhoff related, the map exhibits rotational chaos.

\end{thm*}

We contrast this with the previously known result in this setting, where the condition implying annular chaos is that of
$x_0,x_1$ being \emph{Birkhoff related} \cite{forcing}, that is, for all pairs of neighborhoods $V_0(\varepsilon),V_1(\varepsilon)$ of $x_0$ and $x_1$ respectively, there are orbits which visit  both sets. Since the condition must be verified for every pair of neighborhoods, the property fails to be finitary.

The dissipative version of the previous result follows.

\begin{thm*}

Let $f$ be an annular homeomorphism isotopic to the identity, let $E$ be an essential annulus with $f(E)\subset\mathrm{int}(E)$, such that both boundary components of $f^n(E)$ meet
both disks of a 3-dpn for some pair of fixed points $x_0,x_1$ and some $n\geq 1$. Then $f$ exhibits rotational chaos.

\end{thm*}

The first result is re-introduced below in a more detailed version given by Theorem A, together with a metric version given by Theorem B. For the dissipative maps we have a more detailed statement given by Theorem C. All these results follow the same spirit, and are based on the crucial concept of an \( N \)\emph{-disjoint pair of neighborhoods}, which is the natural generalization of the $3$-dpn just introduced.

In order to show that  our work can be directly applied to examples as it is not purely theoretical,
Section \ref{s.numerics} presents an initial collection of computer-assisted proofs (CAPs) whose conception is based on our theoretical results. These proofs use as programming tools the implementation of interval arithmetic in the $C^{++}$ library CAPD::DynSys (see \cite{compassproof}). This initial sample of CAPs is conducted for several well-known analytic families in both dissipative and Hamiltonian contexts. The numerical implementations are straightforward, as detailed in the aforementioned section, and exhibit significant flexibility, relying solely on basic information about the given maps. We present here only the fundamental aspects of the work \cite{theoappandnumer} in collaboration with Maik Gröger\footnote{Faculty of Mathematics and Computer Science, Jagiellonian University, Poland.} and Maciej Capinski\footnote{Faculty of Applied Mathematics, AGH University of Science and Technology, Krakow, Poland.}, which encompasses many additional scenarios (in both the twist and non-twist regimes) as well as other new theoretical results also covering the phenomenon known as Arnold diffusion. In particular, this new work successfully obtains CAPs for a large set of parameters for the families discussed here, as well as for other cases. Returning to the content of this article, Section \ref{s.numerics} considers the following cases.

\smallskip

In the Hamiltonian case, we deal with the broad variation of the Standard Family given as follows:
$$f_{\textbf{h},\textbf{v}}(x,y)=(x+\textbf{h}(y)\,,\,y+\textbf{v}\,(\sin(2\pi\,(x+\textbf{h}(y)))))$$
where $\textbf{h}$ and $\textbf{v}$ are required to satisfy simple compatibility conditions. This case shows that
the method is indeed flexible, namely, that we can change from one map to another without 
deep analysis of the different situations.

\smallskip

For dissipative maps, we work with the  \emph{Dissipative Standard Family} given by:
$$f_{a,b}(x,y)=\bigl(x+a y,\,b y+\sin(2\pi(x+a y))\bigr),
\qquad a\in\R,\ b\in(0,1).$$
In the excerpt displayed here we fix $a=3$ and the Jacobian $b$ varies from 0.2 to 0.8. We highlight that the Jacobian approaches 1, thereby going beyond the range of the usual strategies, where strong dissipation properties are required. 
Concerning the parameter $a$, it is fixed at 3 as it ensures the existence of the needed pair of fixed points with different rotation numbers in order to apply the theoretical results. For cases below this value, one could just work with powers of the map.

\smallskip

For completeness, let us close the first part of the introduction by listing well-known sufficient conditions and methods in determining positive topological entropy, which can be compared with the ones obtained here.

\begin{itemize}
\item \textbf{Melnikov integral.} This method was already used by Poincar\'{e}, and it shows positive entropy by giving sufficient conditions which imply that a saddle homoclinic connection of an integrable
system becomes a transversal homoclinic intersection when the parameters of the system
vary. It is usually the case that it can only be implemented for parameters \emph{very close} to those of integrable cases. See for instance the nice survey \cite{holmesmeli}.

\item \textbf{Hyperbolic sets and cone fields.} The aim in this case is to find non-trivial hyperbolic sets and is based on the construction of fitted cone fields or, equivalently,
of fitted quadratic-forms. This method has been used in renowned families in the Hamiltonian and dissipative settings (see for instance \cite{bunimovich,laisangyoung}), and usually, in the dissipative setting, it works only in situations where the dissipation is very strong.  For instance, it is often the case in the dissipative context that one tries to make use of the theory of one-dimensional chaotic dynamics, and this approach requires that the Jacobian of the maps are close to zero and
 a detailed study of the differential of the map.

\item \textbf{Nielsen-Thurston Theory.} In the absence of a rich surface topology, as is the case in the annulus, this method requires the existence of a certain configuration of periodic orbits, so that once removed, the remaining dynamics is isotopic to a pseudo-Anosov map. This family of periodic points can be characterized with \emph{braids} (see for instance \cite{braids1,braids2}). The limiting feature for the implementation of this approach in our context is the following: even in the presence of large rotation sets, the topological entropy of these systems can be arbitrarily small (see for instance \cite{paposa}). This means that the periodic orbits needed in order to apply the method could have an arbitrarily large period, and be difficult to detect. 
Another relevant case where Nielsen-Thurston Theory can be used to show positive entropy is for lifts of maps of the $2$-dimensional torus \cite{libremckay}.

\item \textbf{Generic dynamics.} In contrast with the study of specific examples, there is a collection of results where positive entropy is shown to hold in a residual set of area-preserving smooth diffeomorphisms, see \cite{takens,pixton} and recently \cite{lepe}. For example, a generic $C^\infty$ perturbation of the time-one map of the Hamiltonian dynamics associated with the simple pendulum is clearly chaotic. Note that this does not replace the Melnikov integral method since the families we are interested in may very well lie in the complement of such residual set.

\item \textbf{Numerical evidence.} There is a vast literature about numerical evidence of chaos for different meaningful systems inside and outside mathematics. Nevertheless, rigorous proofs are extremely rare.
\end{itemize}

We now proceed to describe our theoretical results in detail.

\subsection{A crucial topological concept.}

Although our main results are topological, we will use the model for the annulus given by 
$$\A=\T^1\times\R\mbox{ where }\T^1=\R/_\sim,\mbox{ and} \ x\sim y\mbox{ if and only if }x-y\in\Z.$$
%Denote by $\textrm{pr}_i:\R^2\to \R,\,i=1,2$ the projections over the first and second coordinate
%respectively. 
Define $\pi:\R^2\to \A$ by $\pi(x,y)=([x],y)$,  where $[x]$ is the equivalence class of $x$ by the relation $\sim$.
Then $\R^2$ together with $\pi$ builds the universal covering of $\A$.
This choice simplifies the definitions and the presentation of the involved techniques.

Let $\homeo$ be the set of orientation and ends preserving
homeomorphisms of $\A$, and consider two fixed points $x_0,x_1$ of $f\in\homeo$.
These points have different \emph{rotation numbers} if
$$(F(x_0)-x_0)-(F(x_1)-x_1)=(\rho,0)\mbox{ with }\rho\in \Z\setminus\{0\}$$
for any lift $F$ of $f$. Such a number does not depend
on $F$ and hence is called \emph{rotational difference} of $x_0,x_1$.
The existence of a pair of fixed points having different rotation numbers for $f$ or some power of it, 
is usually obtained in the context mentioned in the introduction.

\smallskip

Given $x_0$, $x_1$ as above, and $U_0,U_1$ connected neighborhoods of $x_0,x_1$ respectively,
we say that they form a \emph{$N$-disjoint pair of neighborhoods} 
whenever the sets
$$\bigcup_{j=0}^N f^j(U_0)\ \ \ \ \ \ \  ,\ \ \ \ \ \ \ \ \bigcup_{j=0}^N f^j(U_1)$$
are disjoint inessential sets\footnote{Contained inside a topological disk}.
We denote such a pair of neighborhoods by $N$-dpn.

\small
\begin{figure}[htbp]
\centering
\def\svgwidth{.6\textwidth}
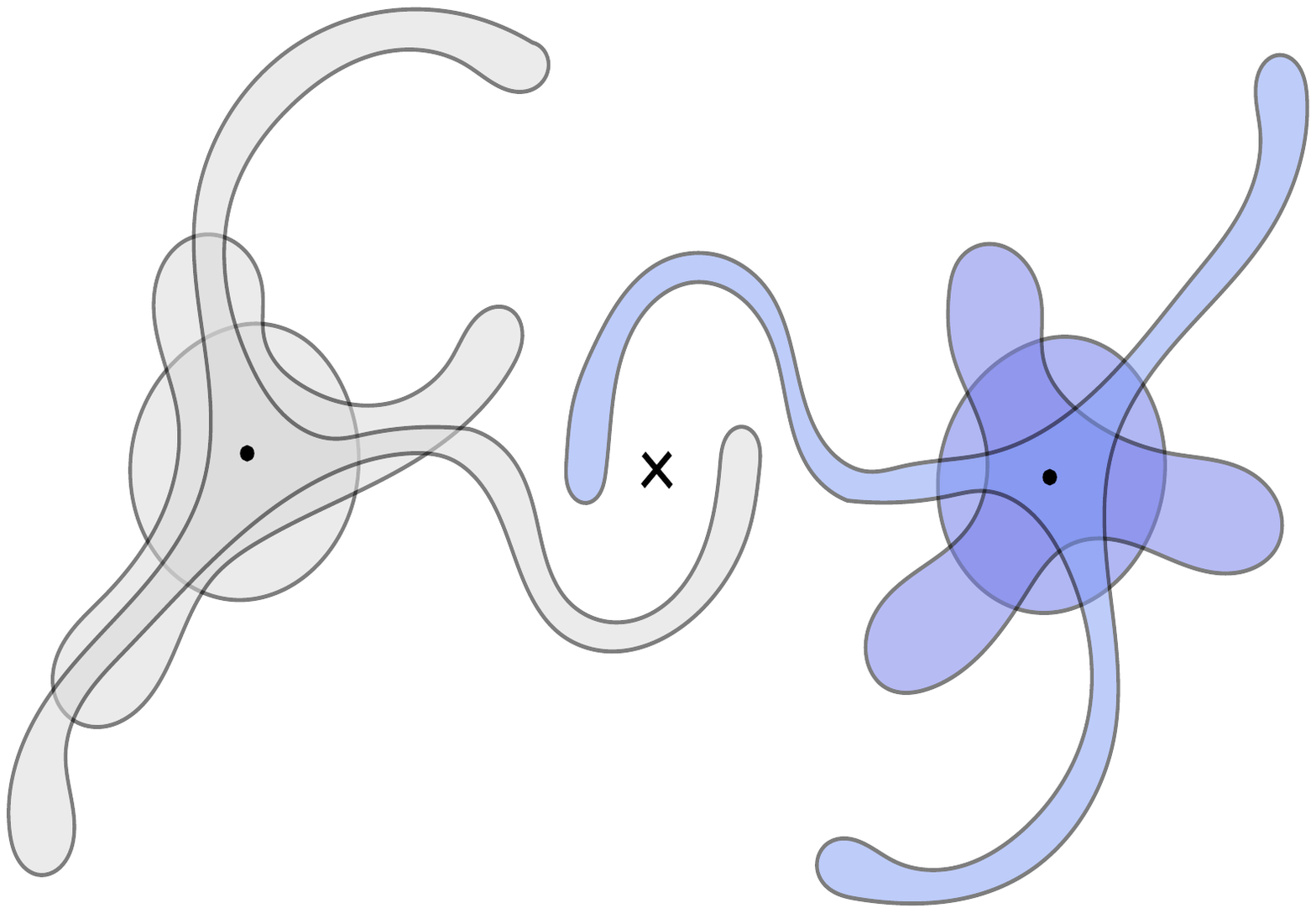
\caption{Representation of a 2-disjoint pair of neighborhoods $U_0,U_1$ for the pair of fixed points $x_0,x_1$. The annulus is represented as the plane minus a single point (denoted by a cross).}
\end{figure}
\normalsize

Finally, we say that two such fixed points $x_0,x_1$ are $N$-\emph{Birkhoff related}
whenever there exists an $N$-dpn $U_0,U_1$ such that both the forward orbit
of $U_0$ intersects $U_1$ and the forward orbit of $U_1$ intersects $U_0$.

\subsection{Results.}

Most of the results are of a purely topological nature, but we also present statements inside metric contexts
as they are relevant in applications.
We state first results related to Hamiltonian settings. These are supported by the techniques developed in this article and also in the characterization of zero entropy maps of the sphere obtained in %by Franks and Handel first in the smooth case, using Nielsen-Thurston theory, and then by Patrice LeCalvez and the second author
 \cite{frankshandel, forcing, lecalveztal}. Denote by $\textrm{Homeo}_{0,\mathrm{nw}}(\A)$ the subset of $\homeo$ consisting of maps whose non-wandering set is equal to $\A$.

\begin{thmA}\label{t.mainconserv}
Let $f\in\textrm{\normalfont{Homeo}}_{0,\textrm{\normalfont{nw}}}(\A)$ and let $x_0,x_1$ be a pair of
fixed points having a rotational difference of $\rho\in\N$. If $x_0,x_1$
are $3$-Birkhoff related, then the map exhibits rotational chaos.

Furthermore, if $\rho= 2$, then the result is valid if $x_0,x_1$
are $2$-Birkhoff related, and if $\rho\ge 3$, then the result is valid if $x_0,x_1$
are $1$-Birkhoff related.
\end{thmA}

\begin{obs}
Let us point out that, in some sense, Theorem A is optimal. It is not difficult to construct examples where the conclusions of the theorem fail to hold when all other hypotheses are kept, but the fixed points are just $2$-Birkhoff related, or if $\rho=2$  the points are $1$-Birkhoff related.  
\end{obs}

Note that in the Hamiltonian setting the state of the art claimed that
the existence of a rotational horseshoe is equivalent to that of an instability region or that of pairs of fixed points being
\emph{Birkhoff related} as defined above. % in \cite{forcing}, which are notions that require arbitrary large traces of orbits in order to be obtained.
With this new result, it is established that the existence of rotational horseshoes
is equivalent to the notion, of finite nature, of the existence of two fixed points with different rotation numbers
$3$-Birkhoff related (or 2 or 1 depending on $\rho$), for $f$ or some power of $f$. 
A similar result obtained with
the same kind of techniques but this time in a metric context follows.
This statement is intended to help with the numerical implementations
as shown in Section \ref{s.numerics} of this paper. 

Define the family of sub-annuli
$$A_L=\pi(\R\times[-L,L])$$
and denote 
$$ N_L(f)=\max_{x\in A_L}\vert \textrm{pr}_2\left(x-f(x)\right)\vert,$$ where 
$\textrm{pr}_2:\A\to\R$ is the projection onto the second coordinate induced in $\A$.

\begin{thmB}\label{c.hamiltoniancasedif}
Let $f\in\textrm{\normalfont{Homeo}}_{0,\textrm{\normalfont{nw}}}(\A)$. Assume there exist $0<L_1<L_2$ such that:
\begin{itemize}
\item[(a)] $f$ has a pair of fixed points $x_0,x_1$ in $A_{L_1}$ with rotational difference given by
a positive integer $\rho$.
\item[(b)] $\bigcup_{i=0}^{r} f^{i}(A_{L_1})\subset A_{L_2}$, where $r=3$ if $\rho=1$, $r=2$ if $\rho=2$ and $r=1$ if $\rho\ge 3$.
\end{itemize}
If some orbit of $f$ visits both connected components of the complement of
$A_{L_2}$, then $f$ exhibits rotational chaos. 
\end{thmB}

\begin{obs}
For numerical computations, note that condition (b) of Theorem B is satisfied whenever $L_2\ge L_1+\max_{1\le i\le r}N_{L_1}(f^i)$ or $L_2\ge L_1+rN_{L_2}(f)$.
\end{obs}

In the dissipative setting, our main result deals with maps leaving invariant an annular continuum (a continuum in $\A$ whose complement consists of precisely two unbounded connected regions) exhibiting non-trivial rotational behavior in its surroundings. In this context, we provide conditions that imply the occurrence of rotational chaos. It is important to note that the mere existence of invariant annular continua with points of differing rotation numbers does not, in itself, entail chaos; thus, additional conditions are necessary. Previously established results indicate that positive topological entropy is guaranteed when the essential annular continuum is an attracting (or repelling) circloid with a non-trivial rotation set, see \cite{paposa, lecalveztal}. Specifically, this refers to an essential annular continuum that does not properly contain other annular continua and acts as either an attractor or a repeller. Theorem C significantly alters this landscape by demonstrating the existence of chaos under considerably weaker assumptions.

\begin{thmC}
Let $F$ be a lift of $f\in\homeo$. Assume that there exists a pair of disjoint essential curves \( \gamma^-, \gamma^+ \) such that, if $E$ is the annulus bounded by them, then  \( f^n(E) \subset \mathrm{int}(E) \) for some \( n \in \mathbb{N} \), and such that $E$ contains two fixed points \( x_0, x_1 \) with rotational difference \( \rho \geq 1 \). If, for some $3$-dpn \( U_0 \) and \( U_1 \) of \( x_0 \) and \( x_1 \), the orbits of both \( \gamma^- \) and \( \gamma^+ \) visit both \( U_0 \) and \( U_1 \), then the following holds.

\begin{itemize}
\item The attractor $\mathcal{K}=\bigcap_{i\in\N} f^i(E)$ contains a unique circloid $\mathcal{C}$.
\item $\rho(\mathcal{C},F)$ is a nontrivial interval of length at least $1/3$.
\item For every rational $r$ in $\rho(\mathcal{C},F)$, there exists a periodic orbit in $\partial \mathcal{C}$ with rotation number $r$.
\item The map $f$ has rotational chaos.
\end{itemize}

Furthermore, if $\rho= 2$, then the result is valid with $U_0,U_1$
forming a $2$-dpn, and if $\rho\ge 3$, then the result is valid with $U_0,U_1$
forming a $1$-dpn.

\end{thmC}

\smallskip
\begin{obs}
It is possible to provide better estimates for the length of the rotation interval in Theorem C, as is done in Theorem~\ref{t.single circloid}. That same result shows that the dissipative hypothesis is not necessary for concluding that an invariant essential annular continuum has a unique circloid, just that both neighborhoods $U_0$ and $U_1$ cross the region.   
\end{obs}

\begin{obs}
Let us point out that, in the same sense as Theorem A, Theorem C is optimal. It is not difficult to construct examples where the conclusions of the theorem fail to hold when all other hypotheses are kept, but $U_0, U_1$ form just a $2$-dpn, or if $\rho=2$ and $U_0, U_1$ form a $1$-dpn.  
\end{obs}

At this point, we can provide a clear idea of the theoretical outcome of this work: the two classical notions used to describe annular chaos, instability regions in the Hamiltonian case and circloids with non-trivial rotation in the dissipative case (both introduced by Birkhoff), are translated by means of Theorem A and Theorem C into concepts of finite nature, which are grounded in basic properties of the maps under consideration.

\smallskip

Finally, let us mention that the results and techniques developed here have a number of theoretical applications, some of which will appear together with the already mentioned computer-assisted proofs in \cite{theoappandnumer}. We highlight three of them. The first one
concerns the so-called diffusion of Hamiltonian periodic annular maps, which is a phenomenon of main importance in mathematics and physics \cite{arnolddiff2half,microion}. The techniques and results here introduced turn out to be important tools for finding a threshold, computed out of basic properties of the prescribed map, for which the existence of an orbit climbing above this threshold
implies the diffusion property for the map and this works in the twist and non-twist regime. 
Our second sort of applications is related to the first one and deals 
with concepts in Rotation Theory of $\mathbb{T}^2$. Again, the results of this article provide the main ingredients
for obtaining bounds on the possible deviations from the rotation set, and also yield quantitative
conditions implying that the rotation set has non-empty interior; the remarkable feature is that these
estimates and conditions depend only on basic properties of the maps, for instance their maximal
displacement and the supremum norm of their differentials.
%As a last application, we obtain bounds on the size of Elliptic Islands for annular dynamics, again depending only on  basic properties of the map.
Such bounds
are fundamental in many of the recent developments in the field, 
%both in the case of bounded deviations 
see \cite{davalos,davalosseg,addas2015uniform,forcing,conejeros2023applications}.
% and in the case of of elliptic Islands (\cite{salvadorkoro,tobiasilhas,korotal,korotalgenus}), and finding them out of explicit data related to the maps is an issue of main importance. 

\subsection{Organization of the paper}
Section \ref{s.prelim} collects the preliminary definitions and several results of a topological and dynamical nature, most of which are standard in the field. 
The core of the article is developed in the subsequent three sections, \ref{s.techresult}, \ref{s.conservative}, and \ref{s.disipative}. Section \ref{s.techresult} presents the more technical results on rotation sets and prime-end rotation numbers associated with annular continua, which are later used to prove Theorems A and B in Section \ref{s.conservative}, and Theorem C in Section \ref{s.disipative}. 
Finally, building on the theoretical results of the paper, Section \ref{s.numerics} provides computer-assisted proofs of the existence of chaos for maps within well-known families.

\subsection{Acknowledgment}
We would like to thank the referees for their suggestions, which made this version much easier to read and allowed us to greatly improve the bounds on rotational difference in Theorem C.

\section{Topological and Dynamical preliminaries}\label{s.prelim}

In this section we establish the required basic notions. Throughout the text, $\N$ denotes the set of positive integers.

\subsection{Topological Preliminaries}

The \emph{annulus} is modeled in this work as
\[
\A = \T^1 \times \R, \quad \text{where } \T^1 = \R / \sim,\ \ x \sim y \ \text{if and only if } x - y \in \Z.
\]
Denote by $\mathrm{pr}_i : \R^2 \to \R$, $i=1,2$, the projections onto the first and second coordinates, respectively. Define $\pi : \R^2 \to \A$ by $\pi(x,y) = ([x],y)$, \textcolor{black}{ where $[x]$ is the equivalence class of $x$ by the relation $\sim$}. 
Then $\R^2$, together with $\pi$, is the universal covering of $\A$.

Note that $\pi$ is a local homeomorphism satisfying $\pi(x+1,y) = \pi(x,y)$ for all $(x,y) \in \R^2$. \textcolor{black}{The Euclidean distance on $\R^2$ induces a distance in $\A$ via this covering, where
$$d_{\A}(([x_1],y_1), ([x_2],y_2))=\min_{x_2'\sim x_2}\sqrt{(x_2'-x_1)^2+(y_2-y_1)^2}.$$}

A set $B$ in $\A$ is said to be \emph{filled} if its complement has no bounded connected component. \textcolor{black}{The \emph{filling of $B$}, denoted $\mathrm{Fill}(B)$, is the union of $B$ with all bounded connected components of its complement. A set $B$ in $\A$ } is said to be \emph{inessential} if it is contained in an open topological disk, otherwise it is \emph{essential}. An essential compact subset of $\A$ separates the ends of $\A$\textcolor{black}{, i.e., the two ends of $\A$ have neighborhoods contained in two different connected components of its complement.} An \emph{open essential sub-annulus} in $\A$ is a topological open annulus that is essential. In this case the inclusion map induces an injective homeomorphism of the respective fundamental groups.
Note that the pre-image of any essential sub-annulus $U$ by $\pi$ gives a connected region of $\R^2$ which is invariant under any horizontal integer translation. Such a region is naturally
called \emph{lift} of $U$ and denoted $\tilde{U}$.

It holds that given a homeomorphism $h$ from an essential sub-annulus $U$ of $\A$ whose image is another essential sub-annulus $V$ of $\A$, we can define a family of lifts 
$H:\tilde{U}\to\tilde{V}$ verifying
$$h\circ\pi=\pi\circ H.$$

We denote by $\homeo$ the space of homeomorphisms in $\A$ which are
isotopic to the identity.

\subsubsection{Rays and Accessible Points.}

An \emph{arc} in a topological space is both a continuous injective function defined on a non-degenerate closed interval $I$ whose image lies in the prescribed topological space and also the image of the curve. We will say that an arc joins points $x$ and $y$ if $x$ and $y$ are the images of the endpoints of $I$.

Given an open and connected set $U\subset\A$, a continuous and injective function $r:[0,+\infty)\to U$ is called a \emph{ray in $U$} if \textcolor{black}{
$$\bigcap_{n>0} \overline{r([n,+\infty))}\subset \partial U.$$} Moreover, we say $x\in\partial U$ is \emph{accessible} from $U$ if there exists some ray $r$ in $U$ so that $\lim_{t\to+\infty}r(t)=x$. In this case, we also say that the ray $r$ \emph{lands} in $x$.

\subsubsection{Annular Continua and Circloids.}
In any topological space a \emph{continuum} is a nonempty compact and connected set.

An essential continuum in $\A$ is then a continuum that separates the two ends of $\A$. An \emph{essential annular continuum} is a filled essential continuum of $\A$, that is, an essential continuum whose complement is given by two connected components, each one being a neighborhood of an end of $\A$.  A \emph{circloid} is an essential annular continuum with the additional property that no proper sub-continuum is also an essential annular continuum. These sets always exist in a prescribed annular continuum, due to a standard Zorn argument, and will be fundamental tools for our purposes. \textcolor{black}{An \emph{essential co-frontier} is  a circloid with empty interior}. In particular, a co-frontier is the boundary of any of the connected components of its complement.

\smallskip

Given an essential annular continuum $\mathcal{K}$, we denote by $\mathcal{U}^+(\mathcal{K})$ the connected component of its complement that contains $\pi(\R\times{[M, +\infty)})$ for some sufficiently large $M>0$. Likewise, we denote
by $\mathcal{U}^-(\mathcal{K})$ the connected component of its complement that contains $\pi(\R\times{(-\infty, -N]})$ for a sufficiently large $N>0$.

A first example of a circloid displaying a complicated topology was given by the so-called \emph{pseudo-circle}, originally constructed by Anderson \cite{andersonpc}; see also \cite{bingpc}. Another important example for our dynamical applications arises from the unstable lamination of a \emph{Smale horseshoe}, which wraps around the annulus $\mathbb{A}$ and can be constructed so that this lamination is a global attractor. This model can even be realized with an Axiom~A system (i.e., one that is hyperbolic on its non-wandering set). By locally modifying a saddle fixed point inside the attractor, one can produce a new dynamical system that creates a source within the global attractor, in a way analogous to the so-called \emph{Derived from Anosov} example. The basin of repulsion of this source is then contained in the essential global attractor given by an annular continuum $\mathcal{A}$. If the construction is carried out carefully, $\mathcal{A}$ remains a circloid. Note that in such a case $\mathcal{A}$ is not given by $\partial \mathcal{U}^-(\mathcal{A}) \cap \partial \mathcal{U}^+(\mathcal{A})$, since the basin of repulsion lies in the complement of this set. In this way, we obtain a dynamically meaningful example of a circloid with non-empty interior, and hence one that is not a co-frontier.

\smallskip

Given two circloids $\mathcal{C}_1, \mathcal{C}_2$  in $\A$, we say that
$\mathcal{C}_2$ is \emph{above} $\mathcal{C}_1$ and denote
$\mathcal{C}_1\preccurlyeq \mathcal{C}_2$ iff $\mathcal{U}^-(\mathcal{C}_1)\subset \mathcal{U}^-(\mathcal{C}_2)$. This defines a partial order on the set of circloids,
and in case $\mathcal{C}_1\preccurlyeq \mathcal{C}_2$ and $\mathcal{C}_1\neq\mathcal{C}_2$
we denote $\mathcal{C}_1\prec\mathcal{C}_2$.

It holds that for any annular continuum $\mathcal{K}$, each of the regions $\mathcal{U}^+(\mathcal{K})$ and $\mathcal{U}^-(\mathcal{K})$ contains in its boundary a unique circloid, denoted $\mathcal{C}^+$ and $\mathcal{C}^{-}$ respectively. Moreover, it is easy to see that whenever $\mathcal{C}^+=\mathcal{C}^-$, then $\mathcal{K}$ contains a unique circloid.

An important property of circloids, which does not hold for general annular continua, is the following: for every point $x \in \partial \mathcal{C}$, one has
\[
x \in \partial\mathcal{U}^-(\mathcal{C}) \cap \partial\mathcal{U}^+(\mathcal{C}).
\]

Although these properties are straightforward to verify, we refer the reader to \cite{jager} for a detailed treatment.

%A circloid $\mathcal{C}\subset \A$ is said to be \emph{compactly generated} if there exists a compact connected subset $\tilde C$ contained in $\pi^{-1}(\mathcal{C})$ such that $\pi(\tilde C)=\mathcal{C}$. In this case one gets that $\bigcup_{i\in\Z}\tilde C+(i,0)=\pi^{-1}(\mathcal{C})$.  We also say that $\tilde C$ \emph{generates} $\mathcal{C}$.

\subsubsection{Moore quotient maps.}

A family $\mathcal{E}$ of continua on $\A$ is an upper semicontinuous family if  the Hausdorff limit of any converging sequence of elements in $\mathcal{E}$ is contained in some element of the family. We recall that, by the Moore decomposition Theorem (\cite{moore}), if $\mathcal{E}$ is an upper semicontinuous family of inessential filled continua of $\A$ such that $\mathcal{E}$ is a partition of $\A$, and if $\sim$ is the equivalence relation in $\A$ defined by $x\sim y$ \textcolor{black}{whenever} $x,y$ belong to the same element in the family $\mathcal{E}$, then the quotient space $\A/_\sim$ is homeomorphic to $\A$. Furthermore, the map that associates to each point its equivalence class can be naturally viewed as a continuous and onto map 
$q:\A\to\A$, which we refer to as the \emph{Moore map} associated to the prescribed upper semicontinuous family. Moreover, the map $q$ lifts to a map $\widetilde q:\R^2\to\R^2$ commuting with the translation $T(\widetilde x)=\widetilde x+(1,0)$. \textcolor{black}{Finally, if $\mathcal{E}$  is an upper semicontinuous family of inessential filled continua of $\A$ such that $\mathcal{E}$ is a partition of some closed set $F\subset \A$, then one can consider the upper semicontinuous family $\mathcal{E}'$ where the elements of $\mathcal{E}'$ are either the elements of $\mathcal{E}$ or the singletons $\{x\}$ with $x$ not in $F$. As $\mathcal{E}'$ is a partition of $\A$, it has a Moore map associated to it. We refer to this map also as the Moore map associated to $\mathcal{E}$.}

%We say that $x\in\partial U$ is \emph{virtually
%accessible} from $U$ if it is contained in an inessential continuum $K\subset\partial U\subset\A$
%for which there is some ray $r$ in $U$
%such that
%$$\bigcap_{t\in\R^+}\textrm{cl}[r(t,+\infty)]\subset K.$$
%Note that the ray does not need to accumulate on $x$.
%In the described situation a set like $K$ is called
%a \emph{support} of the virtually accessible point $x$ and the ray $r$.
%In particular, an accessible point
%is virtually accessible.

\subsection{Dynamical Preliminaries}\label{ss.dynamicalpreliminaries}
Here we state three fundamental concepts for our purposes, namely \emph{Rotation Sets},
\emph{Prime-end Rotation Numbers}, \emph{Rotational Horseshoes}.

\subsubsection{Rotation Sets.}

As already introduced, the set of maps $f:\A\to\A$ that are homeomorphisms in the isotopy class
of the identity is denoted by $\homeo$. The lifts of any element
$f\in\homeo$ form a family of orientation preserving homeomorphisms
in $\R^2$ verifying that $F(x+(1,0))=F(x)+(1,0)$ for all $x\in\R^2$, where $F$
is any of such lifts.
The difference of any two of these lifts is given by an integer horizontal
translation. Fixing one of them, say $F$, allows one to associate a rotation set \textcolor{black}{
$$\rho(F)=\bigcap_{n\ge 1}\overline{\bigcup_{\stackrel{i\ge n}{x\in\R^2}}\frac{\mathrm{pr}_1(F^{i}(x)-x)}{i}}.$$
The set $\rho(F)$ coincides with the set of values $r$ for which there exists an increasing sequence of positive integers $(n_i)$ and a sequence $(x_i)$ with $x_i \in \R^2$, such that 
$$r=\lim_{i\to+\infty} \frac{\mathrm{pr}_1(F^{n_i}(x_i)-x_i)}{n_i}.$$}

Changing the lift of $F$ modifies the rotation set by an integer horizontal translation of $\R^2$.
\textcolor{black}{Also, since $\R^2$ is connected, for every positive integer $i$, the set $\bigcup_{x\in\R^2}\frac{\mathrm{pr}_1(F^{i}(x)-x)}{i}$ is a nonempty interval $I_i$, and a standard argument from rotation theory shows that $I_{pi}$ is a subset of $I_i$ for all $p\in\N$. Thus $I_{ij}$ is a subset of both $I_i$ and $I_j$ for all $i,j\in\N$. From this one deduces that for every $n\in\N$, $\bigcup_{i\ge n}I_i$ is an interval. Thus $\rho(F)$ is the nested intersection of closed intervals, and so it is} always a closed interval (possibly unbounded or empty).

\smallskip

When we have a compact invariant set $X\subset \A$ for an element $f\in\homeo$
it is possible to define $\rho(X,F)$ the rotation set relative to $X$ for any lift $F$ of $f$ \textcolor{black}{ as the set of numbers $r$ for which there exists an increasing sequence of positive integers $(n_i)$ and a sequence $(x_i)$ with $x_i \in \pi^{-1}(X)$, such that 
$$r=\lim_{i\to+\infty} \frac{\mathrm{pr}_1(F^{n_i}(x_i)-x_i)}{n_i},$$
or, equivalently, 
$$\rho(X,F)=\bigcap_{n\ge 1}\overline{\bigcup_{\stackrel{i\ge n}{x\in\pi^{-1}(X)}}\frac{\mathrm{pr}_1(F^{i}(x)-x)}{i}}.$$}

This kind of set is also found in the literature denoted by $\rho_X(F)$. Whenever $X$ is a continuum, it holds by a similar argument as the one presented for the whole rotation set of $F$, that $\rho(X,F)$ is a compact interval. In the case where $X$ is a singleton $X=\{x\}$, we also denote
$\rho(X,F)=\rho(x,F)$. Note that in this situation, $x$ is a fixed point for $f$, \textcolor{black}{since $\{x\}$ is invariant, and thus there exists $p\in\Z$ such that $F(\tilde x)=\tilde x+(p,0)$ for any $\tilde x\in\pi^{-1}(x)$. Hence $\rho(x,F)=p$ must be an integer.} 

We say that two fixed points $x_0,x_1$ of $f$ \emph{have different rotation numbers} whenever $\rho(x_0,F)\neq\rho(x_1,F)$ for a lift $F$ of $f$. This last property does not depend upon the choice of lift $F$ of $f$. We call \emph{the rotational difference of $x_0,x_1$}
the integer number
$$\rho=\rho(x_1,F)-\rho(x_0,F)$$
which we will usually consider positive.

\subsubsection{Prime-End Rotation Numbers.}
Let $\mathbb{S}^2=\A\cup\{-\infty,+\infty\}$ be the usual ends-compactification of $\A$, a space homeomorphic to the $2$-sphere. Let $\mathcal{K}$ be an essential annular continuum in $\A$, and consider the open topological disk
$$D^+=\mathcal{U}^{+}(\mathcal{K})\cup\{+\infty\}\subset \mathbb{S}^2.$$
{\color{black} We briefly recall Carath\'eodory's theory of prime ends for $D^+$, following Section~3 of \cite{korolena}; see also \cite{Caratheodory}.
 
A \emph{cross-cut} of $D^+$ is the image of a simple arc $\gamma:(0,1)\to D^+$ that extends to an arc $\overline{\gamma}:[0,1]\to \mathrm{cl}(D^+)$ joining two distinct points of $\partial D^+$. The set $D^+\setminus\gamma$ has exactly two connected components, each of which is an open topological disk whose boundary meets $\partial D^+\setminus\overline{\gamma}$; these two components are called the \emph{cross-sections} determined by $\gamma$. 

A \emph{chain of cross-sections} is a sequence $(D_n)_{n\geq 1}$ of cross-sections such that $D_{n+1}\subset D_n$ for every $n\geq 1$, and such that the corresponding cross-cuts are pairwise disjoint. Given a cross-section $D$, we say that a chain $(D_n)_{n\geq 1}$ \emph{divides} $D$ if there exists $n_0>0$ such that $D_n$ is a subset of $ D$ for $n>n_0$. A chain $\mathcal{C}$ \emph{divides} a chain $\mathcal{C}'=(D'_n)_{n\geq 1}$ if, for every $n\geq 1$, there exists $m\geq 1$ such that $D_m\subset D'_n$, and two chains are \emph{equivalent} if each one divides the other. Finally, a chain $\mathcal{C}$ is a \emph{prime chain} if it divides every chain that divide $\mathcal{C}$. A \emph{prime end} of $\mathcal{U}^{+}(\mathcal{K})$ is an equivalence class of prime chains, and we denote the set of prime ends by $\mathcal{S}^+$.
 
The disjoint union $D^+\cup\mathcal{S}^+$ is endowed with the topology generated by the open subsets of $D^+$ together with the sets of the form $D\cup E(D)$, where $D$ is a cross-section of $D^+$ and $E(D)$ denotes the set of prime ends whose associated chains divides $D$. The resulting space is called the \emph{prime-end compactification} of $D^+$, and the fundamental result of the theory (see Theorem~3.3 of \cite{korolena}) shows that $\mathcal{S}^+$ is homeomorphic to a circle and $D^+\cup\mathcal{S}^+$ is homeomorphic to the closed unit disk.
 
The compactification also has the following properties
\begin{itemize}
	\item Every ray in $\mathcal{U}^{+}(\mathcal{K})$ landing at a point $x\in\partial\hspace{1pt}\mathcal{U}^{+}(\mathcal{K})$ converges, in the space $D^+\cup\mathcal{S}^+$, to a prime end $\xi\in\mathcal{S}^+$; a prime end arising in this way is said to be \emph{accessible}, and the point $x$ is uniquely determined by $\xi$.
	\item If $r$ and $r'$ are rays in $\mathcal{U}^{+}(\mathcal{K})$ landing at two distinct points $x\neq x'$ of $\partial \mathcal{U}^{+}(\mathcal{K})$, then the corresponding prime ends $\xi$ and $\xi'$ are distinct.
\end{itemize}
Moreover, the set of accessible points is dense in $\partial\mathcal{U}^{+}(\mathcal{K})$, and the set of accessible prime ends is dense in $\mathcal{S}^+$, see Subsection~3.1 of \cite{korolena} for these facts.

If $f\in\homeo$ leaves $\mathcal{K}$ invariant, then the extension of $f$ to $\mathbb{S}^2$ preserves $D^+$ and maps cross-cuts of $D^+$ to cross-cuts and cross-sections to cross-sections. Thus, $f|_{D^+}$ extends to a homeomorphism of $D^+\cup\mathcal{S}^+$. 
The last property implies that, once an orientation of $\mathcal{S}^+$ is fixed, we can define a Poincar\'e rotation number for the restriction to $\mathcal{S}^+$ of the extension of $f|_{D^+}$. As we will need to make this work with a chosen lift $F$ of $f$, we describe a concrete but non-canonical realization of the prime-end compactification, which is the model we work with in the remainder of the article.
 
Endow $\mathbb{S}^2$ with a complex structure. By the Riemann mapping theorem, there exists a conformal homeomorphism $h^*$ from $D^+$ onto the open unit disk $\D\subset\R^2$ such that $h^*(+\infty)=0$. Declaring that a sequence $(x_n)$ in $D^+$ converges to a point $\xi\in\partial\D$ whenever $(h^*(x_n))$ converges to $\xi$, one obtains a compactification of $D^+$ by the circle $\partial\D$. This compactification depends on the choice of $h^*$; however, Carath\'eodory's theory shows that $h^*$ extends to a homeomorphism between $D^+\cup\mathcal{S}^+$ and the closed disk $\overline{\D}$. Thus, this compactification is homeomorphic to the canonical prime-end compactification, and in particular it satisfies the two properties listed above. We can therefore identify $\mathcal{S}^+$ with $\partial\D$.
 
Let $h_+$ be the restriction of $h^*$ to $\mathcal{U}^{+}(\mathcal{K})$, a homeomorphism onto $\D\setminus\{0\}$. Fixing an identification of $\D\setminus\{0\}$ with $\T^1\times(0,1)$ that sends $\partial\D$ to $\T^1\times\{1\}$, we regard $h_+$ as a homeomorphism
$$h_+:\mathcal{U}^{+}(\mathcal{K})\to\T^1\times(0,1),$$
verifying that any lift
$$H_+:\widetilde{\mathcal{U}}^{+}(\mathcal{K})\to \R\times (0,1)$$
commutes with the integer horizontal translations of $\R^2$, that is, with the Deck transformations of $\A$. With this identification, the prime ends of $\mathcal{U}^{+}(\mathcal{K})$ are the points of $\T^1\times\{1\}$, and the properties listed above read as follows: rays in $\mathcal{U}^{+}(\mathcal{K})$ landing at points of $\mathcal{K}$ are sent by $h_+$ to rays landing at points of $\T^1\times\{1\}$, and rays landing at different points of $\mathcal{K}$ are sent to rays landing at different prime ends.
 
Assume now that $f\in\homeo$ leaves $\mathcal{K}$ invariant. The induced homeomorphism
$$f^{\mathrm{end}}_{+}:\T^1\times(0,1)\to \T^1\times(0,1), \qquad f^{\mathrm{end}}_{+}= h_+\circ f\circ (h_+)^{-1},$$
is homotopic to the identity and, by the extension property above, extends to a homeomorphism of $\T^1\times (0,1]$, which we still denote by $f^{\mathrm{end}}_{+}$.
 
Given a lift $F$ of $f$, there exists a unique lift $F^{\mathrm{end}}_{+}$ of $f^{\mathrm{end}}_{+}$ such that any lift $H_+$ of $h_+$ conjugates the restriction of $F$ to $\widetilde{\mathcal{U}}^{+}(\mathcal{K})$ with the restriction of $F^{\mathrm{end}}_{+}$ to $\R\times(0,1)$. The restriction of $F^{\mathrm{end}}_{+}$ to $\R\times\{1\}$ is the lift of an orientation preserving circle homeomorphism, and hence the Poincar\'e rotation number
$$\lim_{n\to +\infty}\frac{\mathrm{pr}_1\left(( F^{\mathrm{end}}_{+})^n(x,1)-(x,1)\right)}{n}$$
exists and does not depend on $x$. This number is called the \emph{upper prime-end rotation number} of $F$ and is denoted by $\rho^{\mathrm{end}}_{+}(\mathcal{K},F)$.
 
The symmetric construction, carried out with the open topological disk $D^-=\mathcal{U}^{-}(\mathcal{K})\cup\{-\infty\}$, produces the corresponding maps $h_-$, $f^{\mathrm{end}}_{-}$ and $F^{\mathrm{end}}_{-}$, and yields the \emph{lower prime-end rotation number} $\rho^{\mathrm{end}}_{-}(\mathcal{K},F)$.
}
 
In \cite{matsumoto,luishernandez}, it is proved that both quantities belong to the rotation set of $\mathcal{K}$ with respect to $F$, that is,
\[
\rho^{\mathrm{end}}_{+}(\mathcal{K},F) \in \rho(\mathcal{K},F) \quad\text{and}\quad \rho^{\mathrm{end}}_{-}(\mathcal{K},F) \in \rho(\mathcal{K},F).
\]

We will also consider prime-end compactifications of annular regions. If $A\subset \A$ is an open essential and bounded sub-annulus of $\A$, then $\partial A$ has two connected components: we denote by $\partial_{+} A$ the one contained in the closure of the connected component of $\A\setminus A$ that is a neighborhood of the upper end of $\A$, and by $\partial_{-} A$ the other one. The filling $\mathcal{K}^{+}$ of $\partial_{+} A$ and the filling $\mathcal{K}^{-}$ of $\partial_{-} A$ are essential annular continua, and
$$A=\mathcal{U}^{+}(\mathcal{K}^{-})\cap\mathcal{U}^{-}(\mathcal{K}^{+}).$$
Proceeding as before, one defines the prime-end compactification of $A$, obtained by adding to $A$ two disjoint circles, the prime-end circles $\T^1_{-}$ and $\T^{1}_+$, with a topology such that the resulting space is homeomorphic to $\T^1\times[0,1]$. As before, if $g\in\mathrm{Homeo}_0(\A)$ leaves $A$ invariant, then it induces an orientation preserving homeomorphism $g^{\mathrm{end}}$ of $\T^1\times [0,1]$ whose restriction to $\T^1\times (0,1)$ is conjugate to $g$ under some homeomorphism $h:A\to \T^1\times (0,1)$ whose lifts commute with the Deck transformations of $\A$.
 
If $G$ is a lift of $g$, then there exists an associated lift $G^{\mathrm{end}}$ of $g^{\mathrm{end}}$ such that, if $H$ lifts $h$, then $H\circ G=G^{\mathrm{end}}\circ H$.
 
One defines the upper (resp.\ lower) prime-end rotation number of $A$, denoted $\rho^{\mathrm{end}}_+(A, G)$ (resp.\ $\rho^{\mathrm{end}}_{-}(A,G)$), by considering the restriction of $G^{\mathrm{end}}$ to $\R\times\{1\}$ (resp.\ $\R\times\{0\}$) and then proceeding as before. One verifies that
$$\rho^{\mathrm{end}}_{-}(A,G)=\rho^{\mathrm{end}}_{+}(\mathcal{K}^{-},G)\quad\mbox{ and }\quad\rho^{\mathrm{end}}_{+}(A,G)=\rho^{\mathrm{end}}_{-}(\mathcal{K}^{+},G).$$

\subsubsection{Rotational Horseshoes.}

A \emph{topological horseshoe} for a homeomorphism $g\in\mathrm{Homeo}(\A)$ is a $g$-invariant compact subset $\Lambda$ such that the restriction of $g$ to $\Lambda$ is semi-conjugated to the Bernoulli shift $\sigma$ on $\Sigma_M=\{1,\hdots, M\}^{\Z}$ for some integer $M\ge 2$, and such that, if $h$ is the semiconjugacy map, and if $w\in\Sigma_M$ is a periodic point for $\sigma$, then $h^{-1}(w)$ contains a periodic point of the same minimal period as $w$. In particular, $h_{\mathrm{top}}(g)>0$, where $h_{\mathrm{top}}(g)$ is the topological entropy of $g$. A \emph{rotational horseshoe} is a horseshoe $\Lambda$ with nontrivial rotational behavior, that is,  if $G$ is a lift of $g$, then $\rho(\Lambda, G)$ is not a singleton. 

A homeomorphism $g \in\mathrm{Homeo}_0(\A)$ is said to \emph{exhibit rotational chaos} if some power of $g$ has a rotational horseshoe.
The existence of rotational horseshoes for $C^0$ annular dynamics has been established under mild assumptions
over the past decade. Results in this direction, such as \cite{paposa, lecalveztal}, have provided key motivation for the present work. 
\section{Results on Rotation Theory of Annular Continua}\label{s.techresult}

This section presents two main technical results on prime-end rotation numbers of invariant annular continua. The first, Theorem~\ref{t.proximal} shows cases where the (topological) proximity of an annular continuum to a prescribed fixed point implies a restriction on one of its prime-end rotation numbers. The second, Theorem~\ref{t.inter} deals with restrictions for the rotational invariants when two circloids $\mathcal{C}^-\prec\mathcal{C}^+$ intersect. 

After stating the two main results as well as a few known results that will be needed,  Subsection~\ref{sub.3.1} contains the proof of the first result, which will be useful in both Section 4 and Section 5, and Subsection \ref{sub.3.2} proves the second result, which will be needed in Section 5. We finish with Subsection~\ref{sub.3.3}, which contains other propositions in a similar spirit both in statement and proof,  potentially useful for future works. As these  are unnecessary for obtaining theorems A, B and C, the whole subsection can be skipped if the reader so chooses. 

For the first result, we introduce some notations. Given an annular continuum $\mathcal{K}$ invariant under
$f:\A\to\A$, a fixed point $x$ of $f$ and an integer $q\ge 1$, we say that $x$ is $q$-proximal to $\mathcal{K}$ whenever there exists a connected open neighborhood $U$ of $x$ such that:
\begin{itemize}
\item  $U$ intersects $\mathcal{K}$,
\item $\bigcup_{i=0}^{q}f^{i}(U)$ is inessential.
\end{itemize}

\begin{thm}\label{t.proximal}
Let $F$ be a lift of $f\in\mathrm{Homeo}_0(\A)$, let $\mathcal{K}\subset\A$ be an invariant annular continuum for $f$ and assume there exists a fixed point $x$ of $f$ in $\mathcal{U}^{+}(\mathcal{K})$ which is $q$-proximal to $\mathcal{K}$.
Then 
\[
\mid \rho_+^{\mathrm{end}}(\mathcal{K},F)-\rho(x,F)\mid \leq\tfrac{1}{q}.
\]
The analogous statement holds for a $q$-proximal fixed point in $\mathcal{U}^{-}(\mathcal{K})$.
\end{thm}

The second main result of this section follows.

 \begin{thm}\label{t.inter}
Let $F$ be a lift of $f\in\mathrm{Homeo}_0(\A)$,  and let $\mathcal{C}^-\prec\mathcal{C}^+$ be two invariant circloids for $f$ 
such that $\mathcal{C}^-\cap\mathcal{C}^+\neq\emptyset$. Then
\[
\rho(\mathcal{C}^-\cap \mathcal{C}^+,F)
=
\{\rho^{\mathrm{end}}_-(\mathcal{C}^+,F)\}
=
\{\rho^{\mathrm{end}}_+(\mathcal{C}^-,F)\}.
\]
\end{thm}

For the proof of these results we will make use of two previously known results. The first is a classical lemma in surface topology. A proof can be obtained using the theory of Brouwer homeomorphisms:

\begin{lemma}[Lemma 4.1 in \cite{koropass}]\label{l.intersectioncontinua}
Let $A,B$ be two inessential continua in $\A$ and $\tilde{A},\tilde{B}$ respective lifts.
Then whenever $\tilde{A}$ meets $\tilde{B}+(k_0,0)$ and $\tilde{B}+(k_1,0)$
with $k_0<k_1$ integers, it holds
$$\tilde{A}\cap\left(\tilde{B}+(k,0)\right)\not =\emptyset \mbox{ for all }k\in\Z\cap[k_0,k_1].$$
\end{lemma}

The second result is a direct consequence of a theorem by Hernand\'ez-Corbato (Theorem 1 from \cite{paperluis}).

\begin{prop}\label{propluis}
Let $F$ be a lift of $f\in\mathrm{Homeo}_0(\A)$, let $\mathcal{K}$ be an invariant annular continuum, and denote by 
$\mathrm{Acc}^+(\mathcal{K})$ the set of accessible points from $\mathcal{U}^{+}(\mathcal{K})$. 
Then one of the following holds.
\begin{itemize}
\item[(a)] For any lift $\tilde{x}$ of any $x\in \mathrm{Acc}^+(\mathcal{K})$ it holds
$$\rho^{\mathrm{end}}_+(\mathcal{K},F)=\lim_{n\to+\infty}\tfrac{1}{n}\mathrm{pr}_1\left(F^n(\widetilde x)-\widetilde x\right),$$
\item[(b)]  For any lift $\tilde{x}$ of any $x\in \mathrm{Acc}^+(\mathcal{K})$ it holds
$$\rho^{\mathrm{end}}_+(\mathcal{K},F)=\lim_{n\to+\infty}\tfrac{1}{n}\mathrm{pr}_1\left(\widetilde x-F^{-n}(\widetilde x)\right).$$
\end{itemize}
\end{prop}
\textcolor{black}{A similar result holds for $\rho^{\mathrm{end}}_-(\mathcal{K},F)$ using the set of  accessible points from $\mathcal{U}^{-}(\mathcal{K})$.} 

\textcolor{black}{A direct consequence is given by the following corollary. 
\begin{cor}
Let $F$ be a lift of $f\in\homeo$ and $\mathcal{K}$ an invariant annular continuum. Then both $\rho^{\mathrm{end}}_+(\mathcal{K},F)$ and $\rho^{\mathrm{end}}_-(\mathcal{K},F)$ belong to $\rho(\mathcal{K},F)$. 
\end{cor} 
\begin{proof}
We show that $\rho^{\mathrm{end}}_+(\mathcal{K},F)$ lies in $\rho(\mathcal{K},F)$, the other case being similar. If item (a) of Proposition~\ref{propluis} holds, then the result is immediate from the definition of $\rho(\mathcal{K},F)$ by taking $\tilde{x}$ lifting a point in $\mathrm{Acc}^+(\mathcal{K})$ . If (b) holds, then one observes that, taking $\tilde{x}$ lifting a point in $\mathrm{Acc}^+(\mathcal{K})$, and defining $\tilde{y}_n=F^{-n}(\tilde{x})$, for all positive integers $n$ the point $\tilde{y}_n$ lifts a point of $\mathcal{K}$, and 
$$\lim_{n\to+\infty}\tfrac{1}{n}\mathrm{pr}_1\left(F^{n}(\widetilde y_n)-\widetilde y_n\right)=\lim_{n\to+\infty}\tfrac{1}{n}\mathrm{pr}_1\left(\widetilde x-F^{-n}(\widetilde x)\right)=\rho^{\mathrm{end}}_+(\mathcal{K},F),$$
so that again, by the definition of rotation set, the result holds. 
\end{proof}
}

We will also need the following lemma on extension of rays in the prime-end compactification:
\begin{lemma}\label{stilline}
Let $f\in\mathrm{Homeo}_0(\A)$, $\mathcal{K}$ be an $f$-invariant annular continuum, and assume that $\sigma:[0,1]\to \A$ verifies that $\sigma\mid_{[0,1)}$ is a ray in $\mathcal{U}^{+}(\mathcal{K})$ landing on a point $x\in\mathcal{K}$. Let $h_{+}:\mathcal{U}^{+}\to\T^1\times(0,1)$ be the homeomorphism associated to the  prime-end compactification of $\mathcal{U}^{+}(\mathcal{K})$, and let $\beta:[0,1]\to\T^1\times(0,1]$ be  the continuous extension of $h_+(\sigma\mid_{[0,1)})$.  If for some $q>0$, $\sigma\cup f^q(\sigma)$ is inessential, then $\beta\cup (f^{\mathrm{end}}_{+})^q(\beta)$ is also inessential.
\end{lemma}
\begin{proof}

Since $\sigma\cup f^q(\sigma)$ is inessential, there exists a proper and injective map $\Gamma:\R\to\A$ whose image is disjoint from $\sigma \cup f^q(\sigma)$, and such that $\lim_{t\to\pm \infty}\Gamma(t)=\mp \infty$. Let $t_1$ be the first value such that $\Gamma(t)\notin \mathcal{U}^{+}(\mathcal{K})$. Then 
$$r:[t_1-1,t_1)\to \mathcal{U}^{+}(\mathcal{K}),\, r(t)=\Gamma(t)$$ is a ray landing on a point $x$ which is neither 
$\sigma(1)$ nor $f^q(\sigma(1))$. 

As discussed in the preliminaries, $h_+(r)$ lands at a prime-end $\xi$ that is neither $\beta(1)$ nor $(f^{\mathrm{end}}_+)^q(\beta(1))$. Thus, the set 
$$\alpha:=h_+(\Gamma(-\infty,t_1)) \cup \{\xi\}$$
is disjoint from $\beta \cup (f^{\mathrm{end}}_+)^q(\beta)$, so 
$\beta \cup (f^q)^{\mathrm{end}}(\beta)$ is contained inside a closed topological disk
of the space $\left(\T^1\times(0,1]\right)\setminus\alpha$. Thus, it is contained in an open topological disk of $\T^1\times(0,1]$, showing the result.
\end{proof}

\subsection{Proof of Theorem \ref{t.proximal}}\label{sub.3.1}

If an essential simple closed curve $\gamma$ is invariant by a  homeomorphism $f\in\mathrm{Homeo}_0(\A)$, and if $F$ is a lift of $f$, then the action of $f$ on $\gamma$ is conjugated to that of a homeomorphism of the circle, and it follows from classical Poincar\'e's rotation number theory that every point in $\gamma$ must have the same rotation number for $F$, which is $\rho(\gamma, F)$. The following lemma is a simpler version of Theorem~\ref{t.proximal}, useful to understand the main idea.

\begin{lemma}~\label{lemmabpatrice}
Let $F$ be a lift of $f\in\mathrm{Homeo}_0(\A)$, $\gamma$ be an $f$-invariant simple essential closed curve in $\A$, and $x$ be a fixed point of $f$ that is $q$-proximal to $\gamma$ for some integer $q\ge 1$. Then 
$$\mid \rho(x,F)-\rho(\gamma,F)\mid\leq \frac{1}{q}.$$
\end{lemma}
\begin{proof}
We may assume, with no loss in generality, that $F(\tilde x)=\tilde x$ whenever $\tilde x$ is a lift of $x$. We may also assume, by a coordinate change if necessary, that $\gamma=\pi(\R\times\{0\})$. If $x$ lies in $\gamma$ then the result follows as all points in $\gamma$ have the same rotation number. If not, let $U$ be a connected neighborhood of $x$ such that $U$ intersects $\gamma$, and such that $U\cup f^q(U)$ is inessential. Let $\sigma\subset U$ be a simple arc joining $x$ to a point $y$ in $\gamma$, and that is disjoint from $\gamma$ except for its endpoint. Let $\widetilde \sigma$ be a lift of $\sigma$, joining $\widetilde x$ to $\widetilde y=(a,0)$. If $F^q(\widetilde y)=(b,0)$, then we claim that  
$$a-1<b<a+1.$$ 

Indeed, assume for a contradiction that $b\geq a+1$. Note first that $b-a$ cannot be an integer as otherwise $\sigma\cup f^q(\sigma)$ is an essential set contained in  $U\cup f^q(U)$. Therefore we can write $b=a+p+\theta$ where $p$ is a positive integer and $0<\theta<1$. Let 
$$\tilde K_1 = \widetilde\sigma\cup F^q(\widetilde\sigma)\ ,\ \tilde K_2 = [a+\theta, a+1]\times \{0\}.$$ One verifies that  $\widetilde K_1 \cap \widetilde K_2$ is empty, as the only intersection points of $\widetilde K_1$ and $\R\times\{0\}$ are $\widetilde y$ and $F^q(\widetilde y)$. But $\widetilde K_1$ intersects both $\widetilde K_2+(-1,0)$ at $(a,0)$ and  $\widetilde K_2+(p,0)$ at $(b,0)$, a contradiction with Lemma~\ref{l.intersectioncontinua}. One shows likewise that $b>a-1$. Hence, by classical arguments of circle dynamics, one obtains that $\rho(\gamma, F^q)\in[-1,1]$. 
\end{proof}

\begin{proof}[Proof of Theorem~\ref{t.proximal}]
The proof is very similar to that of the previous lemma. Let $\widetilde x$ be a lift of $x$, and we assume, with no loss of generality, that $F(\widetilde x)=\widetilde x$. Let $\sigma:[0,1]\to U$ be a simple arc joining $x=\sigma(0)$ to a point $y=\sigma(1)$ in $\mathcal{K}$, and that is contained in $\mathcal{U}^{+}(\mathcal{K})$ except for $y$. Let $h_{+}$ be the inclusion of $\mathcal{U}^+(\mathcal{K})$ into its prime-end compactification. Let $\beta(t)$ be the continuous extension of the arc that satisfies $\beta(t)=h_+(\sigma(t)),\,0\le t<1$. 

As explained in the preliminaries concerning the prime-end compactification, $\xi=\beta(1)$ is an accessible prime-end, and $\beta\cup (f^{\mathrm{end}}_+)^q(\beta)$ is inessential by Lemma~\ref{stilline}. Choosing $\widetilde \beta$ a lift of $\beta$, and $F^{\mathrm{end}}_{+}$ the lift of $f^{\mathrm{end}}_{+}$ associated to $F$, then $F^{\mathrm{end}}_{+}$ fixes $\widetilde \beta(0)$. Let $\widetilde \beta(1)=(a,1)$. One shows, exactly as in Lemma~\ref{lemmabpatrice}, that if $(F^{\mathrm{end}}_{+})^q((a,1))=(b,1)$, then $a-1<b<a+1$, and the result follows.
\end{proof}

\subsection{Proof of Theorem \ref{t.inter}}\label{sub.3.2}

We begin with a lemma concerning the accessibility of points in the intersection of two circloids.

\begin{lemma}\label{lemmainersecaoacessivel}
Let $\mathcal{C}^-\prec\mathcal{C}^+$ be two circloids, and assume for $x\in\A$ that $\{x\}$ is a connected component of 
$\mathcal{C}^{-}\cap\mathcal{C}^{+}$. Then, $x$ is accessible from $\mathcal{U}^{-}(\mathcal{C}^{+})$ and is also accessible by $\mathcal{U}^{+}(\mathcal{C}^{-})$.
\end{lemma}
\begin{proof}
We show that $x$ is accessible from $\mathcal{U}^{-}(\mathcal{C}^{+})$, the other case being analogous.

Let $(\gamma_n)_{n\in\N}$ be a sequence of simple essential loops in $\mathcal{U}^{-}(\mathcal{C}^{-})$
 such that for all $n \in \N$,
\[
\gamma_{n+1} \subset \mathcal{U}^+(\gamma_n),
\]
and such that its Hausdorff limit  is $\partial \mathcal{C}^{-}$. Note that $x$ cannot be an interior point of $\mathcal{C}^{-}$, since, if it was also an interior point of $\mathcal{C}^{+}$ this would contradict the fact that $\{x\}$ is a connected component of $\mathcal{C}^{-}\cap\mathcal{C}^{+}$, and if $x$ were in $\partial \mathcal{C}^{+}$ this would imply that arbitrarily close to $x$ (and thus in $\mathcal{C}^{-}$) there are points of $\mathcal{U}^{+}(\mathcal{C}^{+})$ which contradicts $\mathcal{C}^-\prec\mathcal{C}^+$. Thus $x$ belongs to the limit of $\gamma_n$. 

Choose, for each $l$, a point $x_l$ in $\gamma_l$ that is closest to $x$, and consider, for each $k>0$ and each $l$ such that $x_l\in B(x,1/k)$, the arc $I_{k,l}$ which is the connected component of $\gamma_l \cap \overline{B(x,1/k)}$ containing $x_l$. We may assume, by using a diagonal process and replacing $(\gamma_n)$ with a subsequence, that for each fixed $k$ the sequence of arcs $I_{k,l}$ converges in the Hausdorff topology to some continua $L_k$ that contains both $x$ and a point in $\partial B(x,1/k)$ and is contained in $\mathcal{C}^{-}$. Note also that $I_{k+1,l}$ is a subarc of $I_{k,l}$.

 We claim that for every $k>0$, there exists $n_0:=n_0(k)$ such that, if  $n_0 \le n < m$, then there exists an arc $\sigma_{n,m,k}$ with
\[
\sigma_{n,m,k}\subset B(x,1/k) \cap \mathcal{U}^-(\mathcal{C}^+),
\]
connecting $I_{k,n}$ and $I_{k,m}$.
\medskip

Indeed, assume for a contradiction that there exists $k_0>0$ such that, if $B := B(x,1/k_0)$, then for every $j \in \N$
there exist integers $m(j)>n(j)\ge j$ for which the arcs $I_{k_0,n}$ and $I_{k_0,m}$
cannot be connected by an arc inside
\[
\mathcal{U}^-(\mathcal{C}^+) \cap B.
\]

If the Hausdorff distance between $I_{k_0+1,n}$ and $I_{k_0+1,m}$ is smaller than $\varepsilon<\frac{1}{(k_0+1)^{2}}$, the fact that these two arcs cannot be connected \textcolor{black}{by an arc inside $\mathcal{U}^{-}(\mathcal{C}^{+})\cap B$} implies that for every point $z\in I_{k_0+1,n}$ there exists a point $y(z)\in\mathcal{C}^{+}$ whose distance to $z$ is less than $\varepsilon$. \textcolor{black}{Indeed, by definition of the Hausdorff distance, there exists $z'$ in $I_{k_0+1,m}$ whose distance to $z$ is less than $\varepsilon$. Since the straight-line segment connecting $z$ and $z'$ is contained in $B$, the contradiction hypothesis shows that the segment must not be contained in $\mathcal{U}^-(\mathcal{C}^+)$, and thus intersects $\mathcal{C}^{+}$. One then just take $y(z)$ to be a point in this intersection}.  This implies, since the sequence $I_{k_0+1, n(j)}$ is convergent in the Hausdorff topology, that its limit is also contained in $\mathcal{C}^{+}$. But this limit contains $x$, is not reduced to a point and is also in $\mathcal{C}^-$, a contradiction.

Let then, for each $k\ge 1$, $\alpha_k$ be an arc in $B(x,1/k)\cap\mathcal{U}^{-}(\mathcal{C}^{+})$ joining the point $b_k\in I_{k,n_0(k)}$ to $a_{k+1}\in I_{k,n_0(k+1)}$, and let, for each $k\ge 2$, $\beta_k$ be the subarc of $I_{k-1,n_0(k)}$ joining $a_k$ to $b_k$. Denoting the concatenation of curves as a product, one verifies that the only accumulation point of the curve $$\Gamma=\prod_{k=2}^{\infty}\beta_k\alpha_k$$ is $x$ and that the curve is contained in $\mathcal{U}^-(\mathcal{C}^{+})$. This implies the existence of a ray in $\mathcal{U}^-(\mathcal{C}^{+})$ landing on $x$ whose image is contained in the image of $\Gamma$, showing the result. 
\end{proof}

\textcolor{black}{
\begin{obs}
The hypothesis in Lemma~\ref{lemmainersecaoacessivel} that the connected component of the intersection of the circloids is a singleton is fundamental. One can construct examples of circloids $\mathcal{C}^-\prec\mathcal{C}^+$ where their intersection $C$ is a closed inessential disk, but where both circloids ``spiral'' around $C$. Then no point of $C$ is  accessible from $\mathcal{U}^{-}(\mathcal{C}^{+})$ or by $\mathcal{U}^{+}(\mathcal{C}^{-})$. 
\end{obs}
}

The next result shows that we can, by suitable changes, work under the assumption that the circloids intersect in accessible points.

\begin{prop}\label{p.intercircacc}
Let $\mathcal{C}^- \prec \mathcal{C}^+$ in $\A$ be two circloids such that 
$\mathcal{C}^- \cap \mathcal{C}^+ \neq \emptyset$. Then the following holds:

\begin{enumerate} 
\item The connected components of $\mathcal{C}^- \cap \mathcal{C}^+$ are an upper semi-continuous family of filled inessential continua.

\item If $q:\A\to\A$ is the associated Moore map, then both $q(\mathcal{C}^-)$ and $q(\mathcal{C}^+)$ are circloids and $q(\mathcal{C}^-)\prec q(\mathcal{C}^+)$.

\item The image of any connected component of $\mathcal{C}^- \cap \mathcal{C}^+$ by $q$ is a singleton and therefore accessible both from $\mathcal{U}^{-}(q(\mathcal{C}^+))$ and from $\mathcal{U}^{+}(q(\mathcal{C}^-))$.
\end{enumerate}
\end{prop}
\begin{proof}
\textcolor{black}{ Since circloids are always filled sets, as both connected components of their complement are unbounded, one has that if a set $A$ belongs to a circloid, then its filling also belongs to it.} Thus, if $A$ is a closed set in $C=\mathcal{C}^{+}\cap\mathcal{C}^{-}$, then the filling of $A$ is also contained in this intersection. Therefore all connected components of $C$ are filled continua. Moreover, $C$ has no essential component as any filled essential subcontinua of a circloid must be the whole circloid, and we have assumed that $\mathcal{C}^{-}\not=\mathcal{C}^{+}$. Therefore every connected component of $C$ is a filled inessential continuum. Let $\mathcal{E}$ be the collection of connected components of $C$, and denote by $C_x$ the element of $\mathcal{E}$ that contains $x$. If $(x_n)_{n\in\N}$ is a sequence in $C$ converging to a point $\overline{x}$ (which also belongs to $C$ since $C$ is closed), and $C_{x_n}$ is a  converging sequence in the Hausdorff topology to a set $C_0$, then $C_0$ is connected and contained in $C$, therefore $C_0\subset C_{\overline{x}}$,  which shows that the family $\mathcal{E}$ is upper-semicontinuous. 

To show that $q(\mathcal{C}^{+})$ is still a circloid, note that $q(\mathcal{C}^{+})$ is an essential filled continuum as a Moore map sends filled sets to filled sets and sends essential continua to essential continua. Furthermore, if $D$ is an essential filled subcontinua of $q(\mathcal{C}^{+})$, then $q^{-1}(D)$ must also be an essential filled subcontinua of $\mathcal{C}^{+}$, and thus by minimality $D=q(\mathcal{C}^+)$.  One shows likewise that $q(\mathcal{C}^{-})$ is a circloid, and as both $q(\mathcal{U}^{-}(\mathcal{C}^{-}))=\mathcal{U}^{-}(q(\mathcal{C}^-))$ and $q(\mathcal{U}^{-}(\mathcal{C}^{+}))=\mathcal{U}^{-}(q(\mathcal{C}^+))$, $q(\mathcal{C}^{-})\prec q(\mathcal{C}^{+})$.

Finally, the Moore map is monotone, that is, the preimage of every connected set is connected. Hence, if there exists a non-trivial connected set $A \subset q(C)$, then by definition of $q$, the set $q^{-1}(A)$ is connected and contained in $\mathcal{C}^- \cap \mathcal{C}^+$. Moreover, it must contain at least two distinct elements of $\mathcal{E}$, which is impossible.

% image of any element of $\mathcal{E}$ is a point, and the pre-image of a connected set by $q$ is connected. Also, since $q:\A\setminus C\to\A\setminus q(C)$ is a bijection, the pre-image of any connected component 
%$A$ of $q(\mathcal{C}^- )\cap q(\mathcal{C}^+)$ must be connected and contained in $C$, and as such $q(q^{-1}(A))$ is a singleton, which implies that $A$ is a singleton, and that $q(\mathcal{C}^- )\cap q(\mathcal{C}^+)=q(C)$ is totally disconnected. The accessibility of the points in $q(C)$ follows from Lemma~\ref{lemmainersecaoacessivel} 
\end{proof}

\begin{proof}[Proof of Theorem \ref{t.inter}]
Denote as before $C=\mathcal{C}^{+}\cap\mathcal{C}^{-}$, and $q$ the associated Moore map. 

Note that the map $q$ commutes with $f$ (and $\widetilde q$ commutes with $F$), so there exists an induced $g\in\mathrm{Homeo}_0(\A)$ such that $g(q(x))=q(f(x))$, and a lift $G$ such that $G(\widetilde q(\widetilde x))=q(F(\widetilde x))$. 

Let
$$\Delta:\A\to\R,\,\Delta(y')=\mathrm{pr}_1(G(\widetilde y')-\widetilde y')$$
 where $\widetilde y'$ is any lift of $y'$. As $q(C)$ is a compact subset invariant by $g$ there exists at least one ergodic measure for $g$ supported on $q(C)$. For $\mu$ any such measure, a $\mu$-typical point $y'$ and any lift $\widetilde y'$ of $y'$, if $\rho_\mu=\int \Delta d\mu$, then
$$\lim_{n\to+\infty}\frac{1}{n}\mathrm{pr}_1\left(G^n(\widetilde y')-\widetilde y'\right)=\lim_{n\to+\infty}\frac{1}{n}\mathrm{pr}_1\left(\widetilde y'-G^{-n}(\widetilde y')\right)= \rho_\mu.$$

But since $y'$ is in $q(C)$ and is accessible, Proposition~\ref{propluis} tells us that 
$$\rho^{\mathrm{end}}_-(q(\mathcal{C}^{+}),G)=\rho^{\mathrm{end}}_+(q(\mathcal{C}^{-}),G)=\rho_\mu,$$
and since $\mu$ was arbitrary, there exists a value $\rho_0=\rho_\nu$ for all $g$-invariant ergodic measures $\nu$. 
{From this, using that the set of Borel probability measures supported on $q(C)$ must be a convex and compact set in the weak-$*$ topology, and that the functional assigning for each invariant measure $\overline{\nu}$ the number $\int \Delta d\overline{\nu}$ is linear, one deduces by a standard argument in ergodic theory that for any point $y'\in q(C)$, it holds}
$$\lim_{n\to+\infty}\frac{1}{n}\mathrm{pr}_1\left(G^n(\widetilde y')-\widetilde y'\right)=\lim_{n\to+\infty}\frac{1}{n}\mathrm{pr}_1\left(\widetilde y'-G^{-n}(\widetilde y')\right)= \rho_0.$$  Since  $\widetilde q$ commutes with $T$ and $C$ is compact, $\|\widetilde q(\widetilde y)- \widetilde y\|$ is uniformly bounded in $\pi^{-1}(C)$. One deduces that, for any $\widetilde y$ lift of $y\in C$,
$$
\begin{aligned}
\lim_{n\to\pm\infty}\frac{1}{n} \mathrm{pr}_1\left(F^{n}(\widetilde y)-\widetilde y\right)&=\\
\lim_{n\to\pm\infty}\frac{1}{n} \mathrm{pr}_1\left( F^{n}(\widetilde y)-\widetilde q(F^{n}(\widetilde y))+ G^{n}(\widetilde q(\widetilde y))-\widetilde q(\widetilde y)+ \widetilde q(\widetilde y)- \widetilde y\right)&=\rho_0.
\end{aligned}$$

It remains to be shown that 
$$\rho_0= \rho^{\mathrm{end}}_-(\mathcal{C}^+,F)=\rho^{\mathrm{end}}_+(\mathcal{C}^-,F).$$
Let $h_{-,1}:\mathcal{U}^{-}(\mathcal{C}^{+})\to \T^1\times(0,1)$ be the inclusion map on the prime-end compactification of $\mathcal{U}^{-}(\mathcal{C}^{+})$ and $f^{\mathrm{end}}_{-}$ the homeomorphism of $\T^1\times(0,1]$ induced by $f$. 
Consider a lift $H_{-,1}:\mathcal{U}^-(\mathcal{C}^+)\to \R\times(0,1)$ of $h_{-,1}$ and the induced lift 
$F^{\mathrm{end}}_{-}:\R\times(0,1]\to\R\times(0,1]$ associated to $F$ as introduced in the preliminaries section.
In the same way consider for $g,q(\mathcal{U}^-(\mathcal{C}^+))$ the corresponding maps $g^{\mathrm{end}_-}$, $h_{-,2}, H_{-,2}$ and $G^{\mathrm{end}}_{-}$.

Then the map 
$$Q=H_{-,2}\circ \widetilde q \circ H_{-,1}^{-1}:\R\times (0,1)\to \R\times (0,1)$$
 conjugates the actions of $F^{\mathrm{end}}_-$ and $G^{\mathrm{end}}_-$ in its domain. 
 We would like to extend $Q$ to $\R\times(0,1]$, and to do so, we will use the set of accessible prime ends.

For any lift $\tilde{r}$ in $\tilde{\mathcal{U}}^{-}(\mathcal{C}^{+})$ of a ray which lands on a point $\tilde{x}$ in the complement
of $\pi^{-1}(C)$, it holds that $\tilde{q}(\tilde{r})$ is a lift of a ray in  
$q(\mathcal{U}^{-}(\mathcal{C}^{+}))$ which lands on a point $\tilde{q}(\tilde{x})$, and that $H_{-,1}(\tilde{r})$ is also a ray that lands on a point $(a,1)$ and $H_{-,2}(\tilde{q}(\tilde{r}))$ lands on a point $(b,1)$. The value $b$ is independent of the chosen $\tilde{r}$ as long as $H_{-,1}(\tilde{r})$ lands on $(a,1)$.  Define $\Psi(a)=b$ on all those $a$ such that $(a,1)$ lifts accessible prime ends of $\mathcal{U}^-(\mathcal{C}^+)$
and let $Q^{\mathrm{end}}((a,1))=(\Psi(a),1)$.

As $\Psi$ is nondecreasing, and since for any such $\tilde{r}$, $H_{-,1}(\tilde{r}+(1,0))=H_{-,1}(\tilde{r})+(1,0)$ and the same holds for $\tilde{q}$, we deduce that $\Psi(a+1)=\Psi(a)+1$.  Thus
$\Psi$ induces a function $\psi$ defined in the accessible prime ends $A_{-,1}$ in $\T^1\times\{1\}$ associated to
$\mathcal{U}^{-}(\mathcal{C}^{+})$ having its image contained in the set of accessible prime ends
$A_{-,2}$ in $\T^1\times\{1\}$ associated to $q\left(\mathcal{U}^{-}(\mathcal{C}^{+})\right)$.
Let  $B_{-,2}$ be the subset of $A_{-,2}$ of accessible prime ends that are not associated to a point in $q(C)$. Since $q$ is a bijection from $\A\setminus C$ onto its image, the image of $\Psi$ contains all lifts of prime ends in $B_{-,2}$.

As discussed in Section \ref{s.prelim}, both the sets $A_{-,1},A_{-,2}$ are dense in $\T^1\times\{1\}$. 
The closure of
those elements of $A_{-,2}$ associated to points in $q(C)$ forms a closed set with empty interior; otherwise, one obtains
that $q(C)$ is not totally disconnected, contradicting Proposition \ref{p.intercircacc}. Then, we obtain that
$B_{-,2}$ is also dense in $\T^1$. Thus, as $\Psi$ is nondecreasing, defined on a dense subset and its image is dense, the function $\Psi$ can be extended
to a continuous non-decreasing and onto map
$$\psi:\T^1\to \T^1$$
and its lifts extend $\Psi$ to be a continuous surjection on $\R$.

%The set of accessible prime ends is dense, so $\Psi$ is defined for a dense subset of $\R$. Moreover, since the set of accessible points of $q(\mathcal{C}^{+})$ is dense but $q(C)$ is totally disconnected, the set of accessible point in $q(\mathcal{C}^{+})$ that are not in $C$ is also dense. If $z'$ is such a point, $\beta$ is a ray in $q(\mathcal{U}^{-}(\mathcal{C}^{+}))$ landing on $z'$, then $q^{-1}(\beta)$ also lands on $\mathcal{U}^{-}(\mathcal{C}^{+})$. This shows that the image of $\Psi$ must be also dense.}

Therefore, one can then extend $Q^{\mathrm{end}}:\R\times\{1\}\to\R\times\{1\}$ as a continuous monotone surjection also satisfying
$$Q^{\mathrm{end}}((a+1,1))= Q^{\mathrm{end}}((a,1))+(1,0)$$ 
and verifying that 
$$Q^{\mathrm{end}}\circ F^{\mathrm{end}}=G^{\mathrm{end}}\circ Q^{\mathrm{end}}$$ in the domain of 
$Q^{\mathrm{end}}$. Since $Q^{\mathrm{end}}(z)- z$ is bounded, this implies by classical arguments
of circle dynamics that 
$$\rho^{\mathrm{end}}_-(q(\mathcal{C}^{+}),G)=\rho^{\mathrm{end}}_-(\mathcal{C}^{+},F)=\rho_0.$$
One shows likewise that $\rho^{\mathrm{end}}_+(q(\mathcal{C}^{-}),G)=\rho^{\mathrm{end}}_+(\mathcal{C}^{-},F)=\rho_0$, proving the theorem.
\end{proof}

\subsection{Restrictions on the proximity of invariant continua to pairs of fixed points}\label{sub.3.3}

In this subsection we consider $f\in\mathrm{Homeo}_0(\A)$ having two fixed points with different rotation numbers, and an invariant continuum $\mathcal{K}$, and we wish to show several situations where $\mathcal{K}$ cannot be proximal to both points. We start with a prototypical example:

\begin{lemma}
Let $f\in\mathrm{Homeo}_0(\A)$ be a map with two fixed points $x_0$ and $x_1$ with different rotation numbers, and let $\gamma$ be a simple closed curve invariant by $f$. Then $x_0$ and $x_1$ cannot be both $3$-proximal to $\gamma$.
\end{lemma}
\begin{proof}
This follows directly from Lemma~\ref{lemmabpatrice}, since, if $F$ is any lift of $f$, then $\mid\rho(x_0,F)-\rho(x_1,F)\mid\ge 1$, which means that the intervals 
$$\left[\rho(x_0,F)-\frac{1}{3}, \rho(x_0,F)+\frac{1}{3}\right]\mbox{ and }\left[\rho(x_1,F)-\frac{1}{3}, \rho(x_1,F)
+\frac{1}{3}\right]$$ 
are disjoint.
\end{proof}

We can improve this result if we assume that $x_0$ and $x_1$ are on the same side as $\gamma$.

\begin{lemma}\label{casocurva}
Let $f\in\mathrm{Homeo}_0(\A)$ and $x_0$, $x_1$ two fixed points of $f$ with different rotation numbers. Consider 
a simple closed curve $\gamma$ invariant by $f$. Assume $x_0$ and $x_1$ are in the same connected component of the complement of $\gamma$, and let $U_0,U_1$ form a $2$-dpn for $x_0,x_1$. Then $\gamma$ cannot intersect both $U_0$ and $U_1$.
\end{lemma}
\begin{proof}
Assume, for a contradiction, that $\gamma$ intersects both $U_0$ and $U_1$. We can assume, by a change of coordinates, that $\gamma=\T^1\times\{0\}$. We fix $F$ a lift of $f$ and assume, with no loss in generality, that $\rho(x_0,F)=0$ and that $\rho(x_1,F)=L\ge 1$.

Let $U\subset\A$ be the connected component of $\A\setminus \gamma$ containing
both $x_0$ and $x_1$. As we assumed that $\gamma$ meets both
$U_0,U_1$ it is possible to consider two arcs $j_0,j_1$
contained in $\textrm{cl}(U)=U \cup \gamma$ and an arc $i$ so that:
\begin{enumerate}
\item $j_0\subset U_0$, $j_1\subset U_1$.
\item $j_0$ starts in $x_0$, finishes in $p_0\in\gamma$ and
$j_0\setminus \{p_0\}\subset U$.
\item $j_1$ starts in $x_1$, finishes in $p_1\in\gamma$ and
$j_1\setminus \{p_1\}\subset U$.
\item $i$ is an inessential arc in $\gamma$ joining $p_0,p_1,$ such that, if $\widetilde i$ is a lift of $i$ with an endpoint in $\widetilde p_0$, then $\widetilde i$ is a subarc of the line segment joining $\widetilde p_0$ and $\widetilde p_0+(1,0)$
\end{enumerate}

Set $G=F^2$. Consider the arc $\theta$ given by the concatenation of $j_0\cdot i\cdot (j_1)^{-1}$ and take a lift
$\widetilde{\theta}$ of $\theta$ starting at $\widetilde{x}_0$, a lift of $x_0$, and finishing
at $\widetilde{x}_1$, a lift of $x_1$. As $G(\widetilde{x}_0)=\widetilde{x}_0$ and
$G(\widetilde{x}_1)=\widetilde{x}_1+(2L,0)$ for some integer $L>0$, it holds that $G(\widetilde{\theta})$ meets
$\widetilde{\theta},\widetilde{\theta}+(1,0)$ and $\widetilde{\theta}+(2,0)$. Consider
$\widetilde{j}_0,\widetilde{j}_1,\widetilde{i}$ respective lifts of $j_0, j_1, i$ contained in $\widetilde{\theta}$.
If $\widetilde{p}_0=(a,0)$, and $\widetilde{p}_1=(b,0)$ are the endpoints of $\widetilde{i}$, then $a<b<a+1$. Given two points $(c_1,0), (c_2,0)$, denote by $[c_1, c_2]$ the line segment joining these two points. Let $\psi:\R\to\R$ be such that $G((a,0))=(\psi(a),0)$. Then the following hold

\begin{enumerate}
\item $\psi(a)<a+1$. Indeed, assume for a contradiction this is false. Then, as $j_0\subset U_0$ and $U_0 \cup f^2(U_0)$ is inessential, this implies $G(\tilde{j}_0) \cap
\widetilde{j}_0+(1,0)=\emptyset$, so the point $\widetilde{x}_0+(1,0)$ must be contained in some bounded connected component of $\R^2\setminus \left([a,\psi(a)]\cup\widetilde{j}_0\cup G(\widetilde{j}_0)\right)$. But  $G(\widetilde{j}_0+(1,0))$ connects $\widetilde{x}_0+(1,0)$ to $G((a,0))+(1,0)$ which is contained in the unbounded component of the complement of
    $[a,\psi(a)]\cup\widetilde{j}_0\cup G(\widetilde{j}_0)$, implying that $G(\widetilde{j}_0+(1,0))$
    meets $G(\widetilde{j}_0)\cup\widetilde{j}_0$, which contradicts the definition of a disjoint pair of neighborhoods.

\item $\psi(b)>b-1$: this follows from a similar argument as the one presented for the previous item.
\end{enumerate}

Moreover, it holds:

\begin{enumerate}[resume]

\item $\psi(a)<b$. Indeed, assume for a contradiction this is false. Then $a<b<\psi(a)$, and as $\widetilde{j}_1$
cannot meet $\widetilde{j}_0$ nor $G(\widetilde{j}_0)$, the point $\widetilde{x}_1$ must be contained in some
bounded connected component of $\R^2\setminus \left(\widetilde{j}_0\cup G(\widetilde{j}_0)\cup [a,\psi(a)]\right)$. Hence
$G(\widetilde{j}_1)-(2,0)\cap \R\times\{0\}$ is given by a point contained in $[a,\psi(a)]$ as otherwise
$G(\widetilde{j}_1)-(2,0)$ must intersect $\widetilde{j}_0 \cup G(\widetilde{j}_0)$. Thus
$G(\widetilde{j}_1) \cap \R\times\{0\}$ is given by a point contained in $[a+2,\psi(a)+2]$. Then
$G(\widetilde{i})$ contains both points $G((a,0))$ (due to the first consideration above)
and $(a,0)+(2,0)$ which is impossible as $\widetilde{i}$ lifts an inessential arc in $\gamma$.

\item $\psi(b)>a+2$: the same argument introduced for the case above works
in this one.

\end{enumerate}

But these last two considerations imply that $G(\widetilde{i})$ meets $\widetilde{i}$ and $\widetilde{i}+(2,0)$ which is absurd\textcolor{black}{: $G(\tilde{i})$ would then contain the segment $[a+1,a+2]$, so $f^2(i)$ would be an essential subset of $\A$, contradicting the fact that $f^2$ is a homeomorphism and $i$ is inessential}.
\end{proof}

The last lemma implies the following. 
\begin{prop}\label{maintechnical}
Let $f\in\mathrm{Homeo}_0(\A)$ and $x_0$, $x_1$ be two fixed points of $f$ with different rotation numbers. Consider an invariant annular continuum $\mathcal{K}$ for $f$. Assume that $x_0$ is either in $\mathcal{U}^{+}(\mathcal{K})$ or is accessible from $\mathcal{U}^{+}(\mathcal{K})$ and that the same holds for $x_1$. If $U_0$ and $U_1$ form a $2$-dpn for $x_0,x_1$, then $\mathcal{K}$ cannot intersect both $U_0$ and $U_1$.   
\end{prop}
\begin{proof}
Let $F$ be the lift of $f$ fixing $x_0$. We assume, with no loss in generality, that $\rho(x_1, F)=L\ge 1$. If $x_0$ is accessible from $\mathcal{U}^{+}(\mathcal{K})$, then $\rho^{\mathrm{end}}_+(\mathcal{K},F)=0$ by Proposition~\ref{propluis}, and also $\mathcal{K}$ intersects $U_0$. But the hypotheses on $x_1$ imply that 
$\rho^{\mathrm{end}}_+(\mathcal{K},F)=L$ if $x_1$ is accessible, and, in case that $\mathcal{K}$ intersects $U_1$, then by  Lemma~\ref{lemmabpatrice}, $\rho^{\mathrm{end}}_+(\mathcal{K},F)\ge L-1/2$, a contradiction in both cases. Therefore the result holds if $x_0$ is accessible from $\mathcal{U}^{+}(\mathcal{K})$ and the same must be true if $x_1$ is accessible from $\mathcal{U}^{+}(\mathcal{K})$. 

The case when both $x_0$ and $x_1$ lie in $\mathcal{U}^{+}(\mathcal{K})$ then follows from Lemma~\ref{casocurva} by studying the same situation in the prime-end compactification, in exactly the same way as Theorem~\ref{t.proximal} followed from Lemma~\ref{lemmabpatrice}.
\end{proof}

We finish the section with an application for invariant \textcolor{black}{inessential continua}.

\begin{prop}\label{t.maintechnical0}
Let $F$ be a lift of $f\in\homeo$; $x_0, x_1$ be fixed points of $f$ such that $\rho(x_1,F)-\rho(x_0,F)\geq 1$,  and $U_0, U_1$ form a $3$-dpn for $x_0,x_1$. If $\mathcal{X}$ is an $f$-invariant inessential
continuum, then it is disjoint from either $U_0$ or $U_1$.
\end{prop}

%For proving this result and the subsequent, it will be useful to consider Lemma \ref{l.intersectioncontinua} already proved for instance in \cite{}.

\begin{proof}%[Proof of Theorem A]
Assume for a contradiction that both neighborhoods $U_0,U_1$ meet $\mathcal{X}$.
Consider a pair of arcs $j_0, j_1$ so that:

\begin{enumerate}
\item $j_0\subset U_0$, $j_1\subset U_1$.
\item $j_0$ starts in $x_0$, finishes in $p_0\in\mathcal{X}$ and
$j_0\setminus p_0\subset \A\setminus\mathcal{X}$.
\item $j_1$ starts in $x_1$, finishes in $p_1\in\mathcal{X}$ and
$j_1\setminus p_1\subset \A\setminus\mathcal{X}$.
\end{enumerate}

As $U_0$ and $U_1$ are disjoint, the set
$K=j_0\cup \mathcal{X}\cup j_1$ is inessential and connected. Consider a lift
$\widetilde{K}$ of $K$ and $\widetilde{j}_0,\widetilde{\mathcal{X}},\widetilde{j}_1$ lifts
of $j_0,\mathcal{X},j_1$ contained in $\widetilde{K}$. Further consider $\widetilde{x}_0$
as the lift of $x_0$ in $\widetilde{j}_0$ and $\widetilde{x}_1$ the lift of $x_1$
in $\widetilde{j}_1$. We assume, with no loss of generality, that $F$ is the lift of $f$ that fixes $\widetilde{x}_0$. Let $F^3=G$.

As $G(\widetilde{x}_0)=\widetilde{x}_0$ and
$G(\widetilde{x}_1)=\widetilde{x}_1+(3L,0)$ where $L=\rho(x_1,F)-\rho(x_0,F)$, Lemma \ref{l.intersectioncontinua}
implies that $G(\widetilde{K})$ meets $\widetilde{K},\widetilde{K}+(1,0),\widetilde{K}+(2,0)$ and $\widetilde{K}+(3,0)$. Note that $G(\widetilde{\mathcal{X}})$ must coincide with one of the lifts of $\mathcal{X}$, since $\mathcal{X}$ is an inessential invariant continuum. In particular, either $G(\widetilde{\mathcal{X}})$ is disjoint from $\widetilde{\mathcal{X}}+(1,0)$, or $G(\widetilde{\mathcal{X}})$ is disjoint from $\widetilde{\mathcal{X}}+(2,0)$, and we assume the former holds, the other case being similar. Note also that, by the assumption that $U_0, U_1$ form a $3$-dpn, we have that $G(\widetilde{j}_0)$ is disjoint from all lifts of $\widetilde{j}_1$ and $G(\widetilde{j}_1)$ is disjoint from all lifts of $\widetilde{j}_0$. But $G(\widetilde{K})$ meets $\widetilde{K}+(1,0)$, so one must have that either $G(\widetilde{j}_0)$ intersects $\widetilde{j}_0+(1,0)$ or  $G(\widetilde{j}_1)$ intersects $\widetilde{j}_1+(1,0)$. Since  $G(\widetilde{j}_0)$ intersects $\widetilde{j}_0$ at $\widetilde{x}_0$ and $G(\widetilde{j}_1)$ intersects $\widetilde{j}_1+(3L,0)$ at $\widetilde{x}_1+(3L,0)$, this implies that either $U_0\cup f^3(U_0)$ or $U_1\cup f^3(U_1)$ must be  an essential set, which is impossible since $U_0, U_1$ form a $3$-dpn.
\end{proof}

\section{Proofs of Theorems A and B}\label{s.conservative}
The goal of this section is to obtain Theorems A and B. The proofs combine techniques developed in Section~\ref{s.techresult} and a
topological characterization of \emph{zero-entropy}
homeomorphisms of the $2$-sphere, due to Franks and Handel in the smooth case and
to Le Calvez and the second author in the general case (\cite{frankshandel, forcing}).
\textcolor{black}{First, we need the following consequence of Proposition~54 of \cite{forcing}:}
\textcolor{black}{
\begin{prop}\label{p.54forcing}
Let $G$ lift $g\in \mathrm{\normalfont{Homeo}}_{0,\mathrm{\normalfont{nw}}}(\A)$ with no rotational chaos. If $x_0$ and $x_1$ are fixed points for $g$ such that $\rho(x_0,G)<0<\rho(x_1,G),$ then there exist two disjoint open invariant essential annuli $A_0, A_1$, with $x_0\in A_0$, $x_1\in A_1$ such that, for any $y$ in $\partial A_0\cup\partial A_1$ both the $\omega$-limit and the $\alpha$-limit of $y$ are subsets of $\pi(\mathrm{Fix}(G))$.
\end{prop}}

\textcolor{black}{We also need the following consequence of Theorem A of \cite{lecalveztal}.
\begin{prop}\label{p.teoADuke}
Let $F$ lift $f\in \mathrm{\normalfont{Homeo}}_{0,\mathrm{\normalfont{nw}}}(\A)$ with no rotational chaos. Then, for any $x$ whose forward orbit is bounded and any lift $\widetilde x$ of $x$, the limit  
$$R_F(x):=\lim_{n\to+\infty}\frac{\mathrm{pr}_1( F^n(\widetilde x)-\widetilde x)}{n}$$
exists and is independent of $\widetilde x$. Furthermore, $R_F$ is continuous on the set of points whose forward orbit is bounded
\end{prop}}

As a consequence, we have the following proposition.
\textcolor{black}{
\begin{prop}\label{t.forcing}
Let $f\in \mathrm{\normalfont{Homeo}}_{0,\mathrm{\normalfont{nw}}}(\A)$ with
no rotational chaos. Fix a lift $F$ of $f$ and assume $f$ has a pair
of fixed points $x_0,x_1$ such that $\rho(x_0,F)=0$, and that $\rho(x_1,F)=\rho>0$. Then, there exists an $f$-invariant circloid $\mathcal{C}$ that separates $x_0$ and $x_1$, and such that $\rho(\mathcal{C},F)=\{\rho/2\}$. In particular, for any $y\in \mathcal{C}$ and any $\tilde y$ lift of $y$, 
$$\lim_{n\to+\infty}\frac{1}{n}\mathrm{pr}_1(F^n(\tilde y)-\tilde y)=\frac{\rho}{2}.$$
\end{prop}}
\begin{proof}
First we apply Proposition~\ref{p.54forcing} using $g=f^2$ and $G=F^2-(\rho,0)$, and let $A_0$ and $A_1$ be the $g$-invariant annuli of its conclusion. Since the complement of $\overline{A_0}$ contains at most two essential components, and since $A_1$ is an essential connected set disjoint from $\overline{A_0}$, $x_1$ is in one of these essential connected components.  We will assume, with no loss of generality, that $x_1$ is contained in the connected component of the complement of $\overline{A_0}$ that contains $\T^1\times [M,+\infty)$ for sufficiently large $M$, the other case is similar. Let $\partial^+ A_0$ be the boundary of the connected component $U_1$ of the complement of $\overline{A_0}$ that contains $x_1$, which is an essential continuum that separates $x_0$ and $x_1$, contained in $\partial A_0$ and is disjoint from $A_0$. Since $A_0$ is essential and $\partial^{+} A_0$ is connected, $A_0$ is disjoint from \textcolor{black}{$\mathcal{C}_0=\mathrm{Fill}(\partial^{+} A_0)$. One sees that $\mathcal{C}_0$ is an essential annular continuum, and every point in its boundary is in the intersection of $\partial U_1$ and $\partial A_0$. This implies $\mathcal{C}_0$ is a $g$-invariant circloid separating $x_0$ and $x_1$ as $U_1=\mathcal{U}^{+}(\mathcal{C}_0)$ and $A_0\subset\mathcal{U}^{-}(\mathcal{C}_0)$.}

\textcolor{black}{Note now that $D=\mathcal{C}_0\cup f(\mathcal{C}_0)$ is a bounded $f$-invariant set. It is also connected, as the image of any essential continuum must intersect itself if $f$ has no wandering points. Thus the complement of $D$ has exactly two essential connected components. Also, if $x$ is any point in either $D$ or in an inessential connected component of its complement, i.e. in $\mathrm{Fill}(D)$, then its orbit is bounded and thus, by Proposition~\ref{p.teoADuke}, $R_F(x)$ is well defined and continuous in the essential annular continuum $\mathrm{Fill}(D)$. Also, the $\omega$-limit by $g$ of every point in $\partial D$, by Proposition~\ref{p.54forcing}, is contained in $\pi(\mathrm{Fix}(G))$, which implies that $R_F(x)\equiv \rho/2$ if $x\in\partial D$.}

\textcolor{black}{We claim that if $O$ is the connected component of the complement of $D$ containing $x_0$, then $O$ must be essential. Indeed, if we assume for a contradiction that $O$ is inessential, and if $\widetilde{O}$ is a lift of $O$, since $\rho(x_0,F)=0$ we have that $F(\widetilde{O})=\widetilde{O}$. This implies that, if $y$ in $O$ is recurrent, then $R_F(y)=0$. If $\tilde y \in \widetilde O$ lifts a recurrent point $y$, and $(n_k)$ is an increasing sequence of integers such that $(f^{n_k}(y))$ converges to $y$, then $(F^{n_k}(\tilde y)$ also belongs to $\widetilde O$ for sufficiently large $k$ and so converges to $\tilde y$. This implies that $R_F(y)$ must be zero. Since $f$ has no wandering points, the set of recurrent points is dense in $O$, thus $R_F$ is constant and vanishes in $\overline{O}$, which is impossible since $\partial O\subset\partial D$.}

Thus $x_0$ lies in the complement of $\mathcal{K}=\mathrm{Fill}(D)$, and since $x_0\notin \mathcal{U}^{+}(\mathcal{C}_0)$ and $\mathcal{U}^{+}(\mathcal{K})\subset\mathcal{U}^{+}(\mathcal{C}_0)$, then $x_0\in\mathcal{U}^{-}(\mathcal{K})$.  One shows likewise that $x_1\in\mathcal{U}^{+}(\mathcal{K})$. It suffices then to take any $f$-invariant circloid $\mathcal{C}$ in $\mathcal{K}$.   
\end{proof}

\textcolor{black}{We are ready to provide the proofs of Theorems A and B. We will, until the end of the section, assume that $f,F,x_0,x_1$ and $\rho$ are as in the hypotheses of Proposition~\ref{t.forcing}. In both the proofs of Theorem A and B, we will assume, for a contradiction, that $f$ has no rotational chaos. As such, we can take $\mathcal{C}$ as given in the conclusion of Proposition~\ref{t.forcing}. We will further assume, with no loss in generality, that $x_0\in\mathcal{U}^{-}(\mathcal{C})$, while $x_1\in\mathcal{U}^{+}(\mathcal{C})$.}

\medskip

\begin{proof}[Proof of Theorem A] Note that, as $\rho(\mathcal{C},F)=\{\rho/2\}$, Proposition~\ref{propluis} implies that
$$\rho^{\mathrm{end}}_{+}(\mathcal{C},F)=\rho^{\mathrm{end}}_{-}(\mathcal{C},F)=\rho/2.$$ 

\textcolor{black}{Assume first that $\rho\ge 3$. Then Theorem~\ref{t.proximal} tells us that, if $U_0, U_1$ form a $1$-dpn for $x_0, x_1$, then neither $U_0$ nor $U_1$ can intersect $\mathcal{C}$ which implies that $U_0\subset \mathcal{U}^{-}(\mathcal{C})$ and $U_1\subset \mathcal{U}^{+}(\mathcal{C})$. So, as $\mathcal{U}^{-}(\mathcal{C})$ and $\mathcal{U}^{+}(\mathcal{C})$ are invariant and disjoint, the forward orbit of $U_0$ never meets $U_1$, contradicting the assumption that $x_0$ and $x_1$ are $1$-Birkhoff related.} In the case where $\rho=2$, one observes that in this case if $U_0, U_1$ form a $2$-dpn for $x_0, x_1$, then neither $U_0$ nor $U_1$ can intersect $\mathcal{C}$ and then one argues as before. In the case where $\rho=1$, we observe that if $U_0, U_1$ form a $3$-dpn for $x_0, x_1$, then again neither $U_0$ nor $U_1$ can intersect $\mathcal{C}$, and then apply the same reasoning.
\end{proof}

\smallskip

Before proving Theorem B, let us recall that, for $L>0$, we defined 
$$A_L=\T^1\times[-L,L]\mbox{ and }N_L(f)=\max_{x\in A_L}\vert \mathrm{pr}_2\left(x-f(x)\right)\vert.$$ 

\smallskip

\begin{proof}[Proof of Theorem B]{\color{black} 
  Let us assume the existence of $L_1$ and $L_2$ satisfying items (a) and (b). Let $\sigma:[0,1]\to A_{L_1}, \,\sigma(t)=x_0+t(x_1-x_0)$ be the line segment joining $x_0$ and $x_1$. There exist some minimal $t_0>0$ and some maximal $t_1<1$ such that $\sigma(t_0)$ and $\sigma(t_1)$ belong to $\mathcal{C}$. In particular, $\sigma([0,t_0))\subset \mathcal{U}^{-}(\mathcal{C})$ and $\sigma((t_1,1])\subset \mathcal{U}^{+}(\mathcal{C})$.

Note that, by Theorem~\ref{t.proximal}, since $\rho^{\mathrm{end}}_{-}(\mathcal{C},F)=\rho/2$ and since $\rho/2-\rho(x_0,F)=\rho/2>1/r$, $x_0$ cannot be $r$-proximal to $\mathcal{C}$. In particular, if $U$ is any connected open  neighborhood of $\beta_0=\sigma([0,t_0])$, then $\cup_{i=0}^{r}f^{ i}(U)$ is essential, which implies that $\gamma_0=\cup_{i=0}^{r}f^{ i}(\beta_0)$ is essential. The hypothesis on $L_1$ and $L_2$ imply that $\gamma_0\subset A_{L_2}$. Let $\mathcal{C}_0\subset \mathrm{Fill}(\gamma_0)$ be a circloid. Then, as $\beta_0$ is contained in $\overline{\mathcal{U}^{-}(\mathcal{C})}$, we have that $\mathcal{C}_0\preceq\mathcal{C}$, and $\mathcal{C}$ is contained in $\overline{\mathcal{U}^{+}(\mathcal{C}_0)}$. Since $\T^1\times (-\infty,-L_2)\subset \mathcal{U}^{-}(\mathcal{C}_0)$, we deduce that $\mathcal{C}\subset \T^1\times [-L_2,+\infty)$.

An analogous argument, using that $x_1$ cannot be $r$-proximal to $\mathcal{C}$, shows that  $\mathcal{C}\subset \T^1\times (-\infty, L_2]$. We deduce then that $\mathcal{C}\subset A_{L_2}$, that $\T^1\times(-\infty, -L_2)$ is contained in $\mathcal{U}^{-}(\mathcal{C})$ and that $\T^1\times(L_2,+\infty)$ is contained in $\mathcal{U}^{+}(\mathcal{C})$. But $\mathcal{U}^{-}(\mathcal{C})$ and $\mathcal{U}^{+}(\mathcal{C})$ are invariant and disjoint, a contradiction with the fact that some orbit of $f$ visits both connected components of the complement of $A_{L_2}$.} 
\end{proof}

Before we close the section, let us introduce a proposition which will be used in Section \ref{s.numerics}
(see also \cite{theoappandnumer}) for the computer-assisted proofs
of the existence of chaos of maps inside relevant families. To do so, we consider the two torus 
$\mathbb{T}^2=\R^2/\Z^2$, and for a given
map $f:\T^2\to\T^2$ homotopic to the identity, the vertical displacement function $\phi:\T^2\to\R$ given by 
$\phi(x)=\mathrm{pr}_2(\tilde{f}(\tilde{x})-\tilde{x})$ where $\tilde{x}$ represents $x$
and $\tilde{f}$ lifts $f$.

\begin{prop}\label{p.hamiltonianisnonwandering}
Let $f$ be a homeomorphism of $\T^2$ that preserves a probability measure of full support $\lambda$, let $\hat \pi:\A\to\T^2$ be the canonical projection, and assume that $f$ has an
orientation preserving lift $\hat f$ to $\A$. Further, assume that $\hat f$ has two fixed points with rotational difference $\rho\in\N$, that 
$\textrm{h}_{\textrm{\normalfont{top}}}(\hat f)=0$ and that $\int_{\T^2}\phi\,d\lambda=0$.
Then $\hat f$ is non-wandering. 
\end{prop}
\begin{proof}
By the Birkhoff Ergodic Theorem, the function
 $$\begin{aligned}
\overline\phi(x):=\lim_{n\to+\infty}\frac{1}{n}\left(\sum_{i=0}^{n-1}\phi(f^i(x))\right)=\\
=\lim_{n\to+\infty}\frac{1}{n} \mathrm{pr}_2(\hat f^{n}(\hat x)-\hat x)\mbox{  with } \hat x\in\hat\pi^{-1}(x)
\end{aligned}$$
 is defined for almost all $x$ and satisfies $\int\overline \phi\, d\lambda=0$. Then, either $\overline \phi(x)$ vanishes almost everywhere, or there exist $x_1, x_2 \in \T^2$ such that $\overline\phi(x_1)<0<\overline\phi(x_2)$. But, in the second case, one can find, for every $M>0$, 
 $$\hat x_1\in\hat\pi^{-1}(x_1)\cap\left(\R\times(-\infty,-M]\right),\, \hat x_2\in\hat\pi^{-1}(x_2)\cap\left(\R\times [M, +\infty)\right),$$
  and some $n>0$ such that $\hat f^n(\hat x_1)\in \R\times [M, +\infty)$ and $\hat f^n(\hat x_2)\in \R\times (-\infty, -M]$. This implies that $\A$ is a Birkhoff region of instability for $\hat f$ as defined in \cite{lecalveztal}, and since $\hat f$ has two fixed points with positive rotational difference, Proposition D  of \cite{lecalveztal} implies that $\textrm{h}_{\textrm{top}}(\hat f)>0$, a contradiction with our hypotheses, and therefore $\overline\phi(x)=0$ for almost every $x$. 

\smallskip
 
Assume now for a contradiction that there exists some $\hat y \in \A$ and some $\varepsilon<1/2$ such that $\hat U=B_{\varepsilon}(\hat y)$ is wandering. This implies that $\hat U+(0,i)$ is also wandering for all $i\in\Z$. Let $U=\hat\pi(\hat U)$ and note that $\lambda(U)>0$. Then Birkhoff Ergodic Theorem implies the existence of $x$ in $U$ such that $\overline\phi(x)=0$ and such that the forward orbit of $x$ visits $U$ with frequency at least $\lambda(U)$. Hence, there exists an increasing sequence of integers $(n_j)_{j\in\N}$ such that $f^{n_j}(x)\in U$ and $\lim_{j\to+\infty}\frac{j}{n_j}\ge\lambda(U)$. But, if $\hat x\in\hat\pi^{-1}(x)\cap \hat U$, then for each $j>0$ there exists $i_j\in\Z$ such that $\hat f^{n_j}(\hat x)$ is in $\hat U+(0,i_j)$ and $i_{j_1}\not=i_{j_2}$ if $j_1\not=j_2$ since $\hat U+(0,i)$ are all wandering. Therefore one must have that $\max_{j<k}\{\mid i_j\mid \}\ge (k-1)/2$ which implies that either
$$\limsup_{m\to+\infty}\frac{1}{m}\sum_{n=0}^{m-1}\phi(f^n(x))\ge \lambda(U)/2$$
or
$$\liminf_{m\to+\infty}\frac{1}{m}\sum_{n=0}^{m-1}\phi(f^n(x))\le -\lambda(U)/2$$
must hold. But this contradicts $\overline\phi(x)=0$.   
\end{proof}

\section{Proof of Theorem C}\label{s.disipative}

This section is dedicated to the proof of Theorem C. Note that the statement of Theorem C has two conclusions, one dealing with the existence of a single invariant circloid in the attractor, and then the existence of rotational chaos. For the first conclusion, the dissipative hypotheses can be dropped, as shown by the following result.

\begin{thm}\label{t.single circloid}
Let \( f \in \mathrm{Homeo}_0(\A) \), let $x_0,x_1$ be two fixed points for $f$ with rotational difference $\rho\ge 1$, let $U_0,U_1$ form a $q$-dpn for $x_0,x_1$ and let $\mathcal{K}$ be an $f$-invariant annular continuum. Then, if $\rho>2/q$ and, for $i\in\{0,1\}$,
$$U_i\cap\mathcal{U}^{+}(\mathcal{K})\not=\emptyset\, \hbox{ and } U_i\cap\mathcal{U}^{-}(\mathcal{K})\not=\emptyset,$$
it holds that $\mathcal{K}$ contains a unique circloid $\mathcal{C}$, and its rotation set for any given lift of $f$ contains the interval $[l+1/q,l+\rho-1/q]$ for some integer $l$.
\end{thm}
\begin{proof}
Let us work with  $q=3$ and $\rho=1$, the arguments for the other cases are straightforward adaptations. 

\textcolor{black}{Let $W^{-}$ be the unbounded connected component of the complement of $\overline{\mathcal{U}^{+}(\mathcal{K})}$, and let $U^{+}$ be the unbounded connected component of the complement of $\overline{W^{-}}$. Then Lemma~3.2 of \cite{jager} shows that $\mathcal{C}^{+}=\mathrm{Fill}(\partial U^{+})$ is an invariant circloid, and by direct verification one sees that $\mathcal{U}^{+}(\mathcal{C}^{+})=U^{+}$ and $\mathcal{U}^{-}(\mathcal{C}^{+})=W^{-}$. Furthermore, if $\mathcal{C}\subset\mathcal{K}$ is another invariant circloid, then $\mathcal{U}^{+}(\mathcal{K})$ is a subset of $\mathcal{U}^{+}(\mathcal{C})$ and as $\mathcal{U}^{-}(\mathcal{C})$ is the unbounded connected component of the complement of $\overline{\mathcal{U}^{+}(\mathcal{C})}$, it follows that $\mathcal{U}^{-}(\mathcal{C})\subset W^{-}=\mathcal{U}^{-}(\mathcal{C}^{+})$. Therefore $\mathcal{C}\preccurlyeq\mathcal{C}^{+}$. One defines likewise $U^{-}$ as the unbounded connected component of the complement of $\overline{W^{+}}$, where $W^{+}$ is the unbounded connected component of the complement of $\overline{U^{-}(\mathcal{K})}$. If $\mathcal{C}^{-}=\mathrm{Fill}(\partial U^{-})$, one shows again that }for any circloid contained in $\mathcal{K}$, one has $\mathcal{C}^{-}\preccurlyeq\mathcal{C}$. Let us show then that $\mathcal{C}^{-}=\mathcal{C}^{+}$, which implies that $\mathcal{K}$ has a unique circloid.

Consider the lift $F$ of $f$ such that $\rho(x_0,F)=0$. The following discussion is based on Theorem \ref{t.proximal}.
In case $x_0$ is in $U^{+}$, then $\rho^{\mathrm{end}}_{+}(\mathcal{C}^{+},F)\le 1/3$. If $x_0\in \mathcal{C}^{+}$, then $0\in\rho(\mathcal{C}^{+},F)$ and if $x_0\in\mathcal{U}^{-}(\mathcal{C}^{+})$, then  $\rho^{\mathrm{end}}_{-}(\mathcal{C}^{+},F)\le 1/3$. In any case, $\rho(\mathcal{C}^{+},F)$ has a point in $(-\infty, 1/3]$. The same argument with $x_1$ shows that $\rho(\mathcal{C}^{+},F)$ has a point in $[2/3, +\infty)$, which shows that the length of $\rho(\mathcal{C}^+,F)$ is bounded below by $\frac{1}{3}$.

In order to finish the proof, it is only left to show that $\mathcal{C}^{+}=\mathcal{C}^{-}$ is the unique circloid in 
$\mathcal{K}$. Assume then for a contradiction that $\mathcal{C}^{-}\not=\mathcal{C}^{+}$. We will break the proof into two different cases. Case 1 considers the case when the intersection of $\mathcal{C}^-$ and $\mathcal{C}^+$ is nonempty, whereas Case 2 deals with the case where the intersection is empty.

Assume we are  in Case 1, and let $C=\mathcal{C}^{+}\cap\mathcal{C}^{-}$. We can then apply Theorem \ref{t.inter} and conclude the existence of a value $r$ such that
\[
r = \rho(\mathcal{C}, F)= \rho^{\mathrm{end}}_-(\mathcal{C}^+,F)
= \rho^{\mathrm{end}}_+(\mathcal{C}^-,F).
\]

Now, \textcolor{black}{either $r\notin [-1/3, 1/3]$, or $r\notin [2/3, 4/3]$ since the two intervals are disjoint. If $r\notin [-1/3, 1/3]$, then $x_0$ cannot be in $C$, as otherwise $r=0$, $x_0$ cannot belong to $\mathcal{U}^{+}(\mathcal{C}^{-})$ by Theorem \ref{t.proximal}, and $x_0$ cannot belong to $\mathcal{U}^{-}(\mathcal{C}^{+})$, again by Theorem \ref{t.proximal}. But 
$$\A=C\cup \mathcal{U}^{+}(\mathcal{C}^{-}) \cup \mathcal{U}^{-}(\mathcal{C}^{+}),$$
which is a contradiction since $x_0\in\A$. A similar argument with $x_1$ yields a contradiction if $r$ is not in $[2/3, 4/3]$.} Thus Case 1 cannot hold.

Assume now that we are in Case 2. Hence 
$A:=\mathcal{U}^{+}(\mathcal{C}^{-})\cap \mathcal{U}^{-}(\mathcal{C}^{+})$ 
is a topological annulus, so that $\mathcal{U}^{+}(\mathcal{C}^{-})\cup \mathcal{U}^{-}(\mathcal{C}^{+})=\A$. Consider
the case that $x_0\in\mathcal{U}^{-}(\mathcal{C}^{+})$, the complementary cases can be treated analogously.

Then, $\rho^{\mathrm{end}}_{-}(\mathcal{C}^{+},F)\le \frac{1}{3}$. Again, by Theorem~\ref{t.proximal}, $x_1$ cannot also belong to $\mathcal{U}^{-}(\mathcal{C}^{+})$. Thus $x_1$ is in $\mathcal{U}^{+}(\mathcal{C}^{-})$ and so Theorem
\ref{t.proximal} implies
$$\rho^{\mathrm{end}}_+(\mathcal{C}^-,F)\geq \frac{2}{3}.$$
Thus, $x_0$ cannot belong to $\mathcal{U}^{+}(\mathcal{C}^{-})$.

Let  $\sigma:[0, 1] \to U_0$ be an arc with $\sigma(0)=x_0$, such that $\sigma\mid_{[0,1)}$ is a ray in $\mathcal{U}^{-}(\mathcal{C}^{+})$ landing on a point $y^{+}$ in $\mathcal{C}^{+}$.  Since $x_0\notin \mathcal{U}^{+}(\mathcal{C}^{-})$, there must exist some maximal $t_1\ge 0$ such that $\sigma(t_1)\in\mathcal{C}^{-}$. Denote by $\alpha$ the restriction of $\sigma$ to $(t_1,1)$.  If $\tilde A$ is the lift of the annulus $\mathcal{U}^{+}(\mathcal{C}^{-})\cap \mathcal{U}^{-}(\mathcal{C}^{+})$,  then $\widetilde A$ is a topological disk in $\R^2$ invariant by the translation $T(x,y)=(x+1,y)$. If $\widetilde x_0$ is a lift of $x_0$, $\widetilde \sigma$ the lift of $\sigma$ passing through $\tilde{x}_0$, and $\tilde \alpha$ the lift of $\alpha$ contained in $\widetilde \sigma$, then $\tilde \alpha$ is an oriented  line in $\widetilde A$ and we can denote by $L(\tilde \alpha)$ and $R(\tilde \alpha)$ the connected components of $\tilde{A}\setminus \tilde{\alpha}$ lying to the left and to the right of $\tilde{\alpha}$
respectively.

Note that as $U_0\cup f^3(U_0)$ is inessential it holds that $F^3(\tilde \alpha)$ must be disjoint from $T^p(\tilde \alpha)$ for every integer $p\neq 0$. We claim that 
$$F^3(\tilde \alpha)\subset T(L(\tilde \alpha)).$$ 
Indeed, if this were false, we would find as in the proof of Theorem~\ref{t.proximal} that $\widetilde \sigma\cup F^3(\widetilde \sigma)$ separates $\pi^{-1}(\mathcal{U}^{-}(\mathcal{C}^{+}))$, and that $T\left(\widetilde \sigma\cup F^3(\widetilde \sigma)\right)$ intersects different connected components of this complement, which implies that $\widetilde \sigma\cup F^3(\widetilde \sigma)$ is essential, a contradiction.

Let $h:A\to\T^1\times(0,1)$ be the the homeomorphism associated to the prime-end compactification $\T^1\times[0,1]$ of $A$, $H$ a lift of $h$, $f^{\mathrm{end}}$ the induced homeomorphism on $\T^1\times[0,1]$ and $F^{\mathrm{end}}$ the lift of $f^{\mathrm{end}}$ associated with $F$. Let also $\beta$ be the closure of $h(\alpha)$, an arc joining the two boundaries of $\T^1\times[0,1]$, and $\widetilde {\beta}$ a lift of $\beta$. Then, since $F^3(\tilde \alpha)\subset T(L(\tilde \alpha))$, we obtain that 
$$(F^{\mathrm{end}})^3(\widetilde{\beta})\subset T(L(\widetilde \beta)).$$
 This implies that, if $(a,0)$ is the intersection of $\widetilde\beta$ and $\R\times\{0\}$, then 
 $$(F^{\mathrm{end}})^3((a,0))=(b,0)\mbox{ where }b<a+1.$$
 Again, by a classical argument concerning one-dimensional rotation numbers, this shows that  
 $$\rho^{\mathrm{end}}_+(\mathcal{C}^-,F)\le \frac{1}{3},$$ 
 our final contradiction.
\end{proof}
 
To finish the proof of Theorem C, let us recall Theorem 3.14 of \cite{paposa}.
\begin{thm}\label{thmpaposa}
Let $f\in \mathrm{Homeo}_0(\A)$. Assume that there exists a pair of disjoint essential curves $\gamma^-, \gamma^+ $ such that the annulus \( E \) bounded by them satisfies \( f^n(E) \subset \mathrm{int}(E) \) for some \( n \in \mathbb{N} \).  If there exist four different periodic points $y_1, y_2, y_3$ and $y_4$, each having a different rotation number, and such that all these points are accumulated both by $\cup_{i=0}^{\infty}f^i(\gamma^{+})$ and by $\cup_{i=0}^{\infty}f^i(\gamma^{-})$, then $f$ has rotational chaos.
\end{thm}

\begin{proof}[Proof of Theorem C]\textcolor{black}{Recall that, by hypothesis, $F$ lifts $f\in\homeo$, there exists $n$ in $\N$ and a  pair of disjoint essential curves \( \gamma^-, \gamma^+ \) such that, if $E$ is the annulus bounded by them, then  \( f^n(E) \subset \mathrm{int}(E) \), and such that $E$ contains two fixed points \( x_0, x_1 \) with rotational difference \( \rho \geq 1 \). In particular, if $\mathcal{K}=\bigcap_{i\in\N}f^{i}(E)$, then $\mathcal{K}$ is an invariant attractor and the $\omega$-limit of any point in $E$ is contained in $\mathcal{K}$. Furthermore, we know that there exists some $3$-dpn \( U_0 \) and \( U_1 \) for \( x_0 \) and \( x_1 \) such that the orbits of both \( \gamma^- \) and \( \gamma^+ \) visit both \( U_0 \) and \( U_1 \).} 

The fact that the orbit of $\gamma^{-}$, contained in the invariant set $\subset\mathcal{U}^{-}(\mathcal{K})$, visits $U_0$ implies that $U_0$ intersects $\mathcal{U}^{-}(\mathcal{K})$. Similar arguments show that, for $i$ in $\{0,1\}$ 
$$U_i\cap \mathcal{U}^{+}(\mathcal{K})\not=\emptyset,\, \hbox{and }U_i\cap \mathcal{U}^{-}(\mathcal{K})\not=\emptyset.$$ We may then apply Theorem~\ref{t.single circloid} and deduce that $\mathcal{K}$ contains a single circloid $\mathcal{C}$, whose rotation set is not a singleton, and whose boundary lies in $\partial \mathcal{U}^{+}(\mathcal{K})$ as well as in $\partial \mathcal{U}^{-}(\mathcal{K})$. It follows that 

$$\partial \mathcal{C}\subset \left(\bigcap_{j=0}^{\infty}\overline{\bigcup_{i=j}^{\infty}f^{i}(\gamma^{-})}\cap \bigcap_{j=0}^{\infty}\overline{\bigcup_{i=j}^{\infty}f^{i}(\gamma^{+})}\right).$$
In \cite{koropeki}, it is shown that any circloid $\mathcal{C}$ which is invariant under $f \in \homeo$ with lift $F$, and whose rotation set $\rho(\mathcal{C},F)$ is a non-trivial interval contains, for every rational number $\frac{p}{q} \in \rho(\mathcal{C},F)$, a periodic point $x$ of $f$ realizing this rotation number. That is,
\[
F^q(\tilde{x}) = \tilde{x} + (p,0)
\]
for any lift $\tilde{x}$ of $x$.

Moreover, in \cite{koropass}, it is shown that such a periodic point $x$ can be chosen in the boundary of the circloid $\mathcal{C}$ (recall that these continua may have non-empty interior, as explained in the preliminaries). 

Our result then follows from Theorem~\ref{thmpaposa}.
\end{proof}

\section{Computer-Assisted Proofs}\label{s.numerics}

As mentioned in the introduction, the aim of this paper is to reformulate the topological theory of annular chaos---developed since the work of Poincar\'e and Birkhoff---into a framework that can be effectively tested for a given map using computational methods. In this section, we present a first series of Computer Assisted Proofs (CAPs) based on the results introduced here.
We show the most basic extracts of a work in preparation \cite{theoappandnumer}, where we deal
with some well-known analytic families of maps.  The point we want to emphasize here is that the software required is simple and can be easily adapted
from one case to another. The tool for these CAPs is the open source library CAPD::DynSys for the $C^{++}$ language \cite{compassproof}, which is used for the validation of numerical finite orbits in our systems. These validations
are based on the interval arithmetic method. We would like to thank 
 Maik Gr\"oger\footnote{Faculty of Mathematics and Computer Science,
Jagiellonian University, Krakow, Poland. } and Maciej Capinski\footnote{Faculty of Applied Mathematics, AGH University of Science and Technology, Krakow, Poland.}, who were instrumental in constructing the first versions of the software used here, and who co-author the above mentioned article in preparation.

\smallskip

The first class of examples where one can try to apply the results isgiven by analytic families of diffeomorphisms of $\A$, 
and some of these cases are discussed in what follows. However, it should be noted that another important class of applications is Poincar\'{e}-return maps of flows in the Hamiltonian and dissipative cases, and we expect to explore those in the near future.

\medskip

Let us explain what is shown in the following lines. For the Hamiltonian case first and then for the dissipative one, we describe how the implemented software works and show some pictures that help to visualize what is done. Then, some lists
displaying the first examples for which we obtained the CAPs are introduced. The visualizations were
made using Python whereas the programs themselves were coded in $C^{++}$, as it is needed to work
 with the library CAPD::DynSys\footnote{http://capd.ii.uj.edu.pl}. The programs were run on typical commercially available machines. %For instance, with specifications: Intel i5 processor with a memory of 4gb RAM. 

 \subsection{Hamiltonian case.}

Let us now explain the Hamiltonian case, for which Theorem B is used. The CAPs
in this case are much simpler, as only the analog of point (1) of the dissipative case needs to be implemented, 
and, moreover, the points whose orbits are to be validated belong to a fixed horizontal line.
The structure of the needed program in this case is the following.

\begin{enumerate}

\item For the maps tested here, we choose two fixed points $x_0,x_1$ which lie inside
the projection of the band $\R\times[-1,1]$ and have a rotational difference of 2, and taking $L_2=5$ one obtains that $N_{L_2}(f)\leq 1$ as required by the conditions of Theorem B.  Call $\gamma^+,\gamma^-$ the upper and lower
boundary components of $A_{L_2}$.
Hence, we try to find an orbit going from one of the connected components of the complement of $A_{L_2}$ to the other one. This orbit is first numerically estimated and then validated working with interval arithmetic, using
the CAPD::DynSys library. 

%\item  An spaced mesh of points in $\gamma^+$ and $\gamma^-$ is taken and forwardly numerically iterated. The image of these points is drawn light blue and light red
%respectively, unless they lie in the lower (respectively upper) connected component of the complement of $A_{L_2}$, in which case they are drawn in dark blue and dark red respectively.

\end{enumerate}

The numerics shown here were calculated for maps that are variations of the Standard Family (S.F.). The S.F.
is given by those maps having lifts
$$F_{a,b}(x,y)=(x+a\,y\,,\,y+b\sin(2\pi\,(x+a\,y))),\ a\in\R,$$ 
and the variations here are given by those maps having lifts 
$$F_{\textbf{h},\textbf{v}}(x,y)=(x+\textbf{h}(y)\,,\,y+\textbf{v}\,(\sin(2\pi\,(x+\textbf{h}(y))))),$$
where
\begin{itemize}
\item $\textbf{h}\in C^{\infty}(\R),\, [-1,1]\subset \textbf{h}([-1,1])$,
\item $\textbf{v}\in C^{\infty}(\R),\ \int_{[0,1]}\textbf{v}(\sin(2\pi\,x))\,dx=0$,
\item $\textbf{h}$ lifts a map of $\T^1$.
\end{itemize}
It turns out that these maps contain all the possible analytic diffeomorphisms in $\A$ which lift
maps of $\T^2$ having 0 mean
vertical displacement. We write them in this way
so that they can be regarded as variations of elements of the S.F.

\smallskip

For simplicity, we will instead of presenting $\textbf{v}$, present a function $w \in C^{\infty}(\R)$, and take $\textbf{v}(x)$ to be $w(x)-\int_{[0,1]}w(\sin(2\pi\,x))\,dx$, thus ensuring that $\int_{[0,1]}\textbf{v}(\sin(2\pi\,x))\,dx=0$. 
The conditions for $\textbf{h}$ and $\textbf{v}$ imply the existence of the required fixed points $x_0,x_1$,
one lying in the projection of $y\in \textbf{h}^{-1}(-1)$ and the other in the projection of $y\in \textbf{h}^{-1}(1)$, and
having a rotational difference of 2. The condition on the mean value of $\textbf{v}$, namely that $\int_{[0,1]}\textbf{v}(\sin(2\pi\,x))\,dx=0$ is needed
due to the following proposition.
\begin{prop}
Given any $f_{\textbf{h},\textbf{v}}$ as above, then either every point is non-wandering or $f_{\textbf{h},\textbf{v}}$ has a  rotational horseshoe.
\end{prop}
\begin{proof}
There are two cases to consider. If  $\textbf{h}(x+1)=\textbf{h}(x)$, then  $f_{\textbf{h},\textbf{v}}$ lifts a map of $\T^2$ homotopic to the identity that preserves Lebesgue measure, and satisfies the hypotheses of Proposition~\ref{p.hamiltonianisnonwandering}, in which case we are done. If not, then $\textbf{h}(x+1)=\textbf{h}(x)+p$ for some non-zero integer $p$, in which case $f_{\textbf{h},\textbf{v}}$ lifts a homeomorphism of $\T^2$ homotopic to the Dehn twist $(x,y)\to(x+py, y)$. In this case, one can define a vertical rotation interval for the map (see \cite{doeff1997rotation,addas2002existence}), and the Lebesgue rotation number, by the choice of $\textbf{v}$, is null. Then either the rotation interval is nondegenerate, in which case $f_{\textbf{h},\textbf{v}}$ has positive topological entropy (see, for instance,  sec. 5.2 in \cite{nunes2022transitivity}) or it is just $\{0\}$, in which case every orbit of $f_{\textbf{h},\textbf{v}}$ is bounded (see \cite{addas2012dynamics}) and the dynamics is non-wandering.  
\end{proof}

The previous proposition tells us that in order to show the existence of rotational horseshoes for $f_{\textbf{h},\textbf{v}}$, we can just try to apply Theorem B assuming that the involved dynamics is non-wandering, because otherwise we are already assured positive entropy. Note that, if one studies return maps to Poincar\'e's annular sections for Hamiltonians in two degrees of freedom restricted to compact energy surfaces, then Theorem B can be applied directly, as the return map preserves a finite measure of full support, which ensures that the map is non-wandering. 

\smallskip

Therefore, for the introduced setting, in order to show the existence of chaos it suffices to find
orbits going from the line $y=-5$ to the line $y=5$ (or vice versa). We have tested this for several different elementary functions $\textbf{v}$, and for two different functions $\textbf{h}$, one linear and the other a trigonometric function, which destroys the twist condition. Note that, while there exists a vast literature for the Standard Family (see for instance \cite{mackay2020hamiltonian,gelfreich1999proof}), the arguments there are usually specific to that family, and each variation such as those considered here would require a huge effort in order to be treated by similar methods.

The following pictures are taken from the visualization of the algorithm implemented in Python. 
For the map of the family given by
$$\textbf{h}(y)=\sin(2\pi\,y)\ \ \ ,\ \ \ w(x)=x(1-x),$$
Figure \ref{cons1} shows the initial lines,
Figure \ref{cons1b} shows the first iterate of the lines,
Figure \ref{cons2} an intermediate iteration of the lines for which the conditions
of Theorem B are still not fulfilled, whereas the last picture shows an iterate
of them for which these conditions are fulfilled, and hence
it is possible to infer the existence of chaos.

\begin{figure}[h]
\begin{minipage}{0.35\linewidth}
\includegraphics[width=1\linewidth]{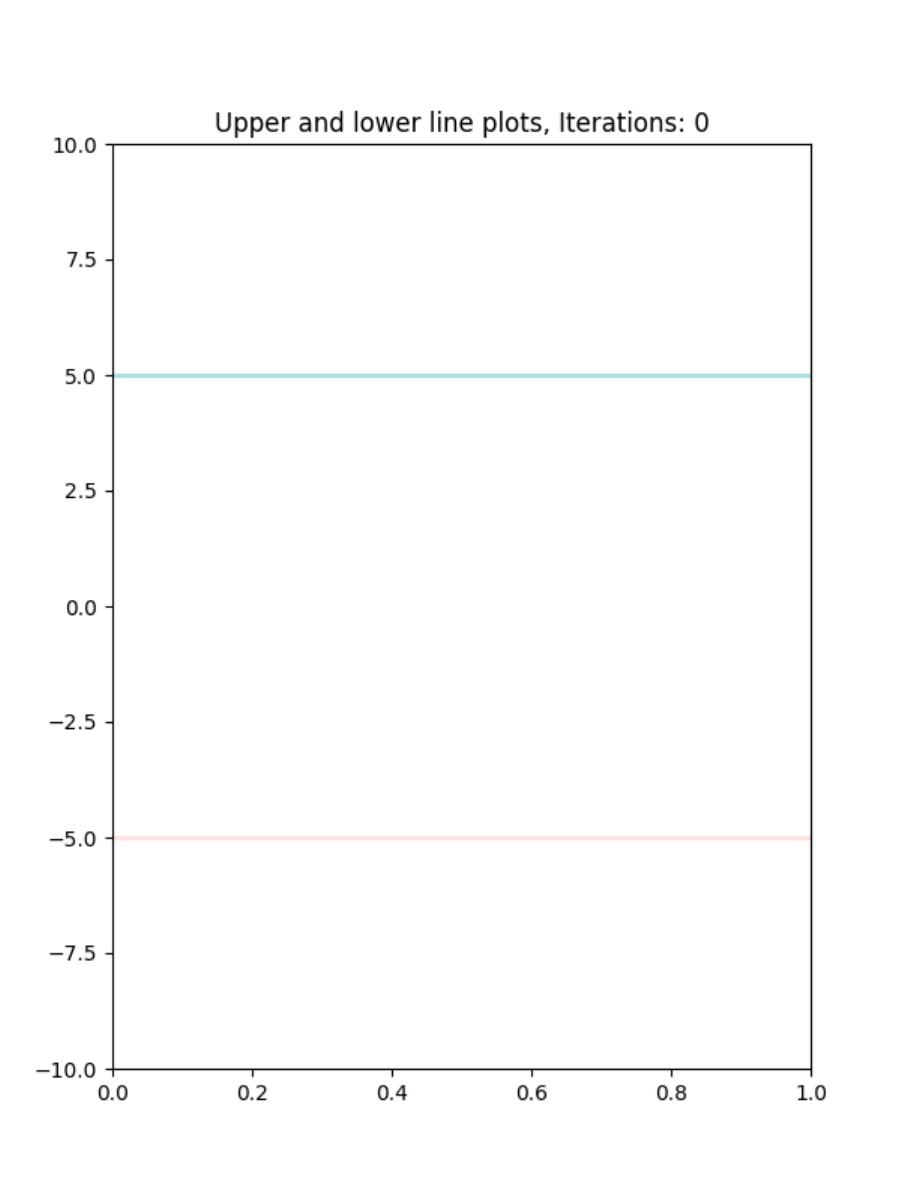}
\caption{}\label{cons1}
\end{minipage}
\begin{minipage}{0.35\linewidth}
\includegraphics[width=1\linewidth]{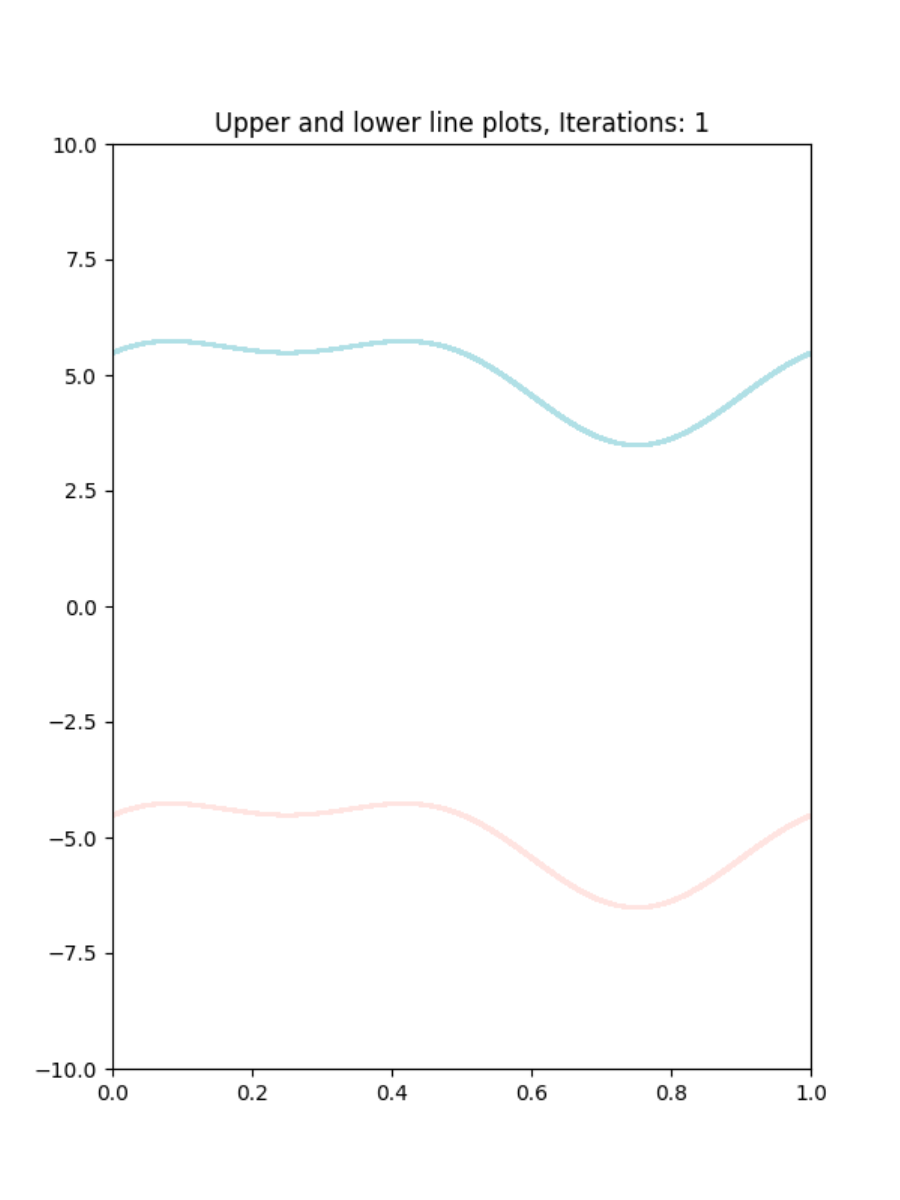}
\caption{}\label{cons1b}
\end{minipage}
\end{figure}
\begin{figure}[h]
\begin{minipage}{0.35\linewidth}
\includegraphics[width=1\linewidth]{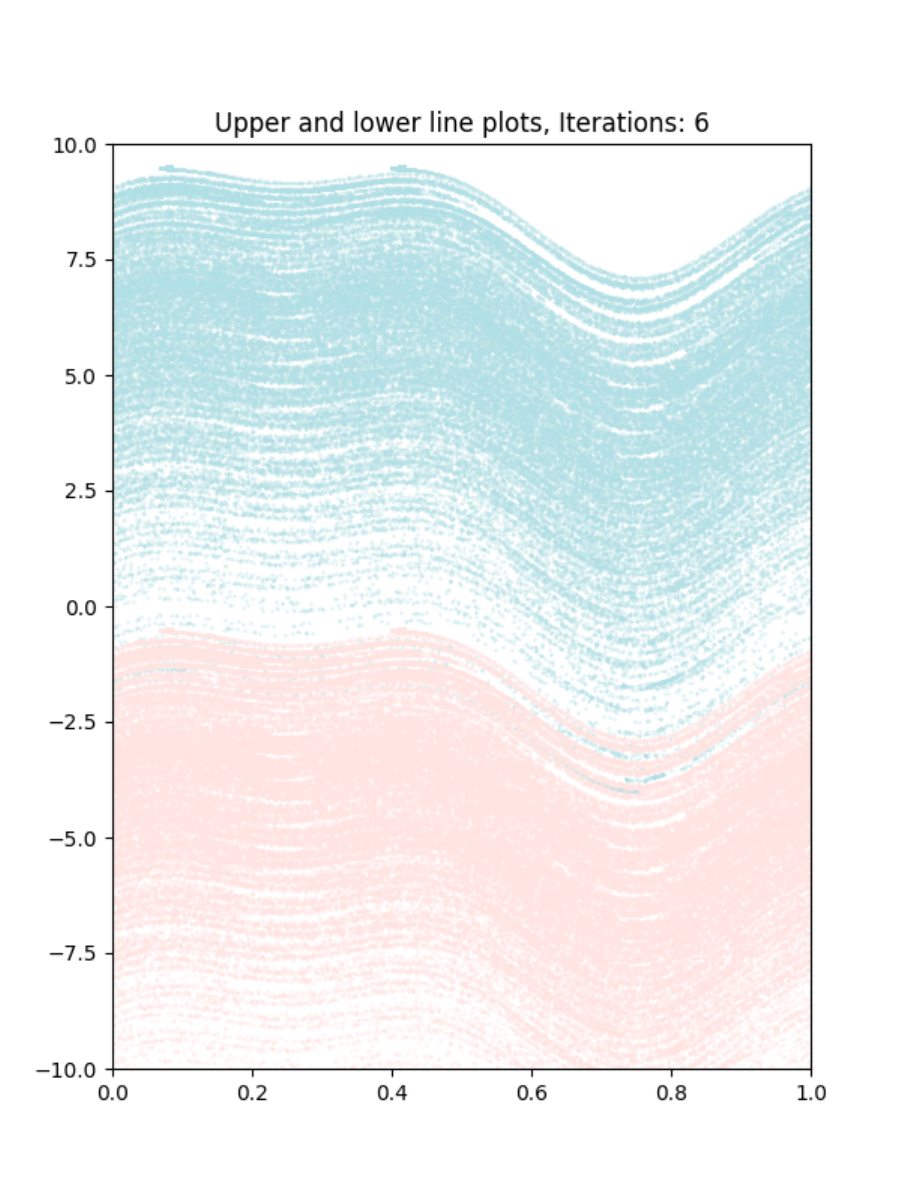}
\caption{}\label{cons2}
\end{minipage}
\begin{minipage}{0.35\linewidth}
\includegraphics[width=1\linewidth]{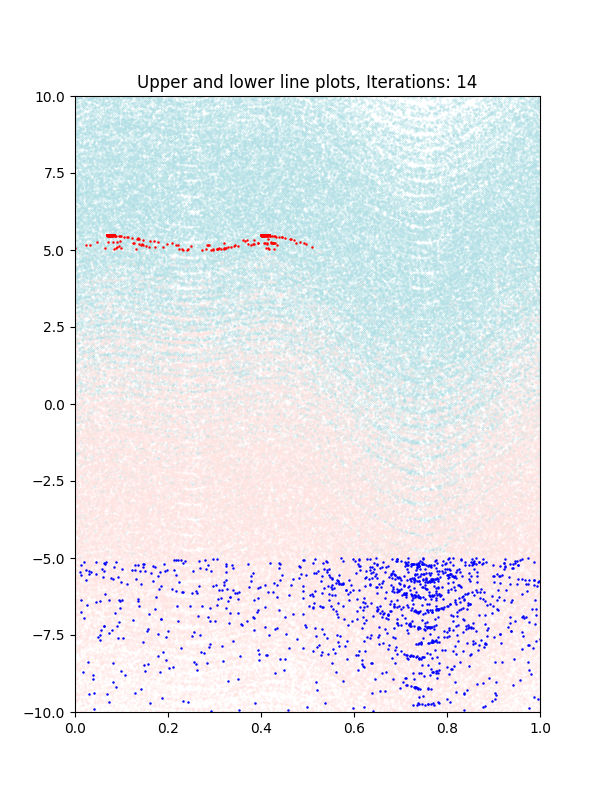}\label{cons3}
\caption{}
\end{minipage}
\end{figure}

The CAPs are in this case straightforward as we only deal with item (1):
we just find a numerical orbit going from $y=-5$ to $y\geq 5$ and then validate it
using Interval Arithmetics by means of the CAPD::DynSys library. 

\smallskip
 
Tables \ref{tb:tableconservative1} and \ref{tb:tableconservative2}  show several situations where the procedure described above gave CAPs of the existence of chaos. It deals with variations of the Standard Family, and its purpose is also to show the flexibility of the method. The tables list the specific maps that were studied as well as the number of iterates needed to find an orbit going from one connected component of the complement of $A_5$ to the other. 
Table \ref{tb:tableconservative1} deals with twist maps while Table \ref{tb:tableconservative2} deals with the non-twist cases.
The first case of Table \ref{tb:tableconservative1} is the classical
Standard Map. 

\large
\begin{table}[ht]
\begin{centering}
\begin{minipage}{0.43\textwidth}
\begin{tabular}{|c|c|c|}\hline
$\textbf{h}(y)$ & $\textbf{w}\,(x)$ & \textbf{It.}\\ \hline
\ \ $y$\ \ &\ \ $x$\ \ &\ \ 12 \ \ \\\hline
%\ \ $y$\ \ &\ \ $0.6\,x$\ \ &\ \ 43 \ \ \\\hline
\ \ $y$\ \ &\ \ $x\,(1-x)$\ \ &\ \ 14\ \ \\\hline
\ \ $y$\ \ &\ \ $\tan(x)$\ \ &\ \ 11 \ \ \\\hline
%\ \ $y$\ \ &\ \ $\cos(2\pi\,x)$\ \ & \ \ 17\ \ \\\hline
\ \ $y$\ \ &\ \ $3\,\ln(x+2)$\ \ & \ \ 11\ \ \\\hline
\ \ $y$\ \ &\ \ $e^x-1$\ \ &\ \ 9 \ \ \\\hline
\end{tabular}
\caption{\textbf{Twist}}
\label{tb:tableconservative1}
\end{minipage}
\hfill
\begin{minipage}{0.45\textwidth}
\begin{tabular}{|c|c|c|}\hline
$\textbf{h}(y)$ & $\textbf{w}\,(x)$ & \textbf{It. }\\ \hline
\ $\sin(2\pi\,y)$\ &\ \ $x$\ \ &\ \ 14 \ \ \\\hline
%\ $\sin(2\pi\,y)$\ &\ \ $0.6\,x$\ \ &\ \ 23 \ \ \\\hline
\ $\sin(2\pi\,y)$\ &\ \ $x\,(x-1)$\ \ &\ \ 8\ \ \\\hline
\ $\sin(2\pi\,y)$\ &\ \ $\tan(x)$\ \ &\ \ 9 \ \ \\\hline
%\ $\sin(2\pi\,y)$\ &\ \ $\cos(2\pi\,x)$\ \ & \ \ 14\ \ \\\hline
\ $\sin(2\pi\,y)$\ &\ \ $3\,\ln(x+2)$\ \ & \ \ 10\ \ \\\hline
\ $\sin(2\pi\,y)$\ &\ \ $e^x-1$\ \ &\ \ 9 \ \ \\\hline
\end{tabular}
\caption{\textbf{Non-twist}}
\label{tb:tableconservative2}
\end{minipage}
\end{centering}
\end{table}

\newpage
\normalsize

\subsection{Dissipative case}

In the dissipative setting we treat here the Dissipative Standard Family $f_{a,b}$, having as lifts the maps
$$F_{a,b}(x,y)=\left(x+a\, y\,,\,b\,y+\sin(2\pi\,(x+a\,y))\right)\,,\ a\in\R,\,b\in(0,1).$$

\smallskip

%\begin{itemize}
%
%\item \emph{Variations on the Disipative Standar Family} given by:
%$$f_{g,b}(x,y)=(x+g(y)\,,\,b\,y+\sin(2\pi\,(x+g(y))))\,,\ g\in C^{\infty}(\R),\,b\in(0,1)$$
%In case $g(y)=a\,y$ we obtain the usual D.S.F. This case and all others were $g$ is strictly monotone induce \emph{twist %maps} diffeomorphisms, where extra theoretical tools  are available \cite{lecalveztesis,crovisierthesis}.
%See \cite{laisangyoung} and reference therein for a previous treatment of similar
%families. Let us refer to this one, in short, as V.D.S.F. .
%\medskip
%
%\item The \emph{Two Dimensional Arnold Family} given by:
%$$f_{a,w,b}(x,y)=(\hat{f}_{a,w}(x)+y\,,\,b\,(\hat{f}_{a,w}(x)-x+y))\,\mbox{ were }$$
%$$a,w\in\R,g\in C^{\infty}(\R),b\in(0,1),\mbox{ and }\hat{f}_{a,w}(x)=x+a\sin(2\pi\,x)+w$$
%which turns the 1D Arnold Family of $\mathbb{S}^1$ given by $\hat{f}_{a,w}$
%into a dissipative annular diffeomorphism. When the Jacobian of these maps are small one approaches the well studied 1D-%case. The simulations here show that the numerical evidence is not restricted to small jacobians. Let us refer to this family %by 2D-A.F. .
%
%\end{itemize}
%\smallskip

The idea is to show numerically, but formally, that  the hypotheses of Theorem C 
are fulfilled for some prescribed maps and hence obtain a proof of the existence of a rotational horseshoe for the given diffeomorphism. 
In the following, we describe a general procedure, which can be done for any $f\in\homeo$ having a global annular attractor
carrying fixed points with different rotation numbers.

\begin{enumerate}

\item Two fixed points $x_0,x_1$ with a rotational difference of $\rho$ are chosen together
with a small positive $r$. Then, we try to find four points $x_0^+,x_0^-$ in $B(x_0,r)$ and
$x_1^+,x_1^-$ in $B(x_1,r)$\footnote{The distance chosen here is  the one induced by the $\max$ norm for the sake of simplicity in what concerns the following lines.} so that the backward orbits of $x_0^+,x_1^+$ visit the region above an essential circle $B$ whose forward orbit
lies below itself, and the backward orbits of $x_0^-,x_1^-$ visit the region below an essential circle  $B'$ whose forward orbit lies above itself.

\item Then, we need to check that the points $x_0,x_0^-,x_0^+$ can be included in a connected neighborhood
$U_0\subset B(x_0,r)$ and the points $x_1,x_1^-,x_1^+$ can be included in a connected neighborhood
$U_1\subset B(x_1,r)$ so that $U_0,U_1$ form an $N$-dpn, with $N$ being 3 if $\rho=1$, 2 if $\rho=2$, 
$1$ if $\rho\geq 3$.
It suffices to show that there exist arcs $\sigma_0, \sigma_1$, where $\sigma_i$ contains the points $x_i,x_i^-,x_i^+$ for $i\in\{0,1\}$, and such that both $\bigcup_{j=0}^{N}f^j(\sigma_0)$ and $\bigcup_{j=0}^{N}f^j(\sigma_1)$ are inessential and disjoint, as one could just then take some neighborhoods $U_0$ and $U_1$ of $\sigma_0$ and $\sigma_1$, respectively.

%\item
%Backward numerical iterations of a grid of points in the boundaries of $U_0$ and $U_1$ are performed. A numerical %verification of the conditions given in Corollary C happens when points in the boundary of both  neighborhoods reach, %under such iterations, the upper region of $\gamma^+$ and the lower region of $\gamma^-$, providing evidence of %positive topological entropy. Light blue and light red is used for the iterations of the boxes. When these iterations reach %the upper and lower connected complements of the complement of the annulus bounded by $\gamma^{+}$ and $ \gamma^{-}$, the color of the points is displayed in dark blue and dark red respectively.

\end{enumerate}

If these two conditions are fulfilled, we have the existence of rotational chaos for the map $f$, as a consequence of
Theorem C.

\smallskip

In what follows, pictures are displayed representing the process for the element of the D.S.F. given
by $a=3,\,b=0.8$. Figure \ref{dis1} shows the small neighborhoods around the fixed points $x_0,x_1$ having sky-blue and light-red colors, and the lines
$B,B'$ given in this case by the projection of
$y=6$ and $y=-6$, respectively. In this case $N=3$ as the rotational difference of $x_0,x_1$ is 12. We need that the neighborhoods centered at $x_0,x_1$ form a $1$-dpn 
and both reach the lower component of $B'$ and
the upper component of $B$ under backwards iterations. The Figures \ref{dis2}
and \ref{dis3}
show intermediate backwards iterates of these neighborhoods having the same mentioned colors as far as they don't reach the mentioned lines. The last
one, \ref{dis4}, shows the change of color when the conditions of Theorem C are numerically fulfilled, namely, when
the neighborhoods reach under backwards iterations the two lines.

\begin{figure}[h]
\begin{minipage}{0.35\linewidth}
\includegraphics[width=1\linewidth]{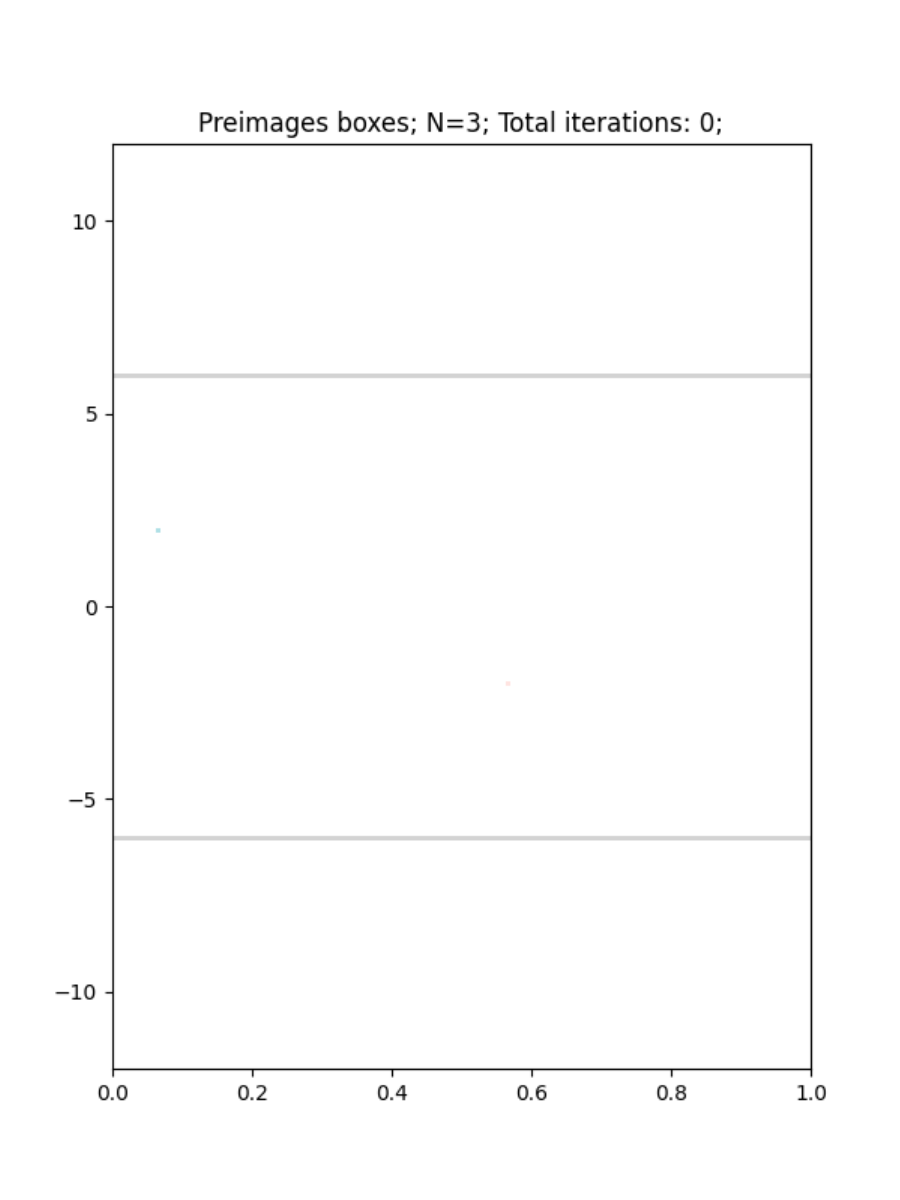}
\caption{}\label{dis1}
\end{minipage}
\begin{minipage}{0.35\linewidth}
\includegraphics[width=1\linewidth]{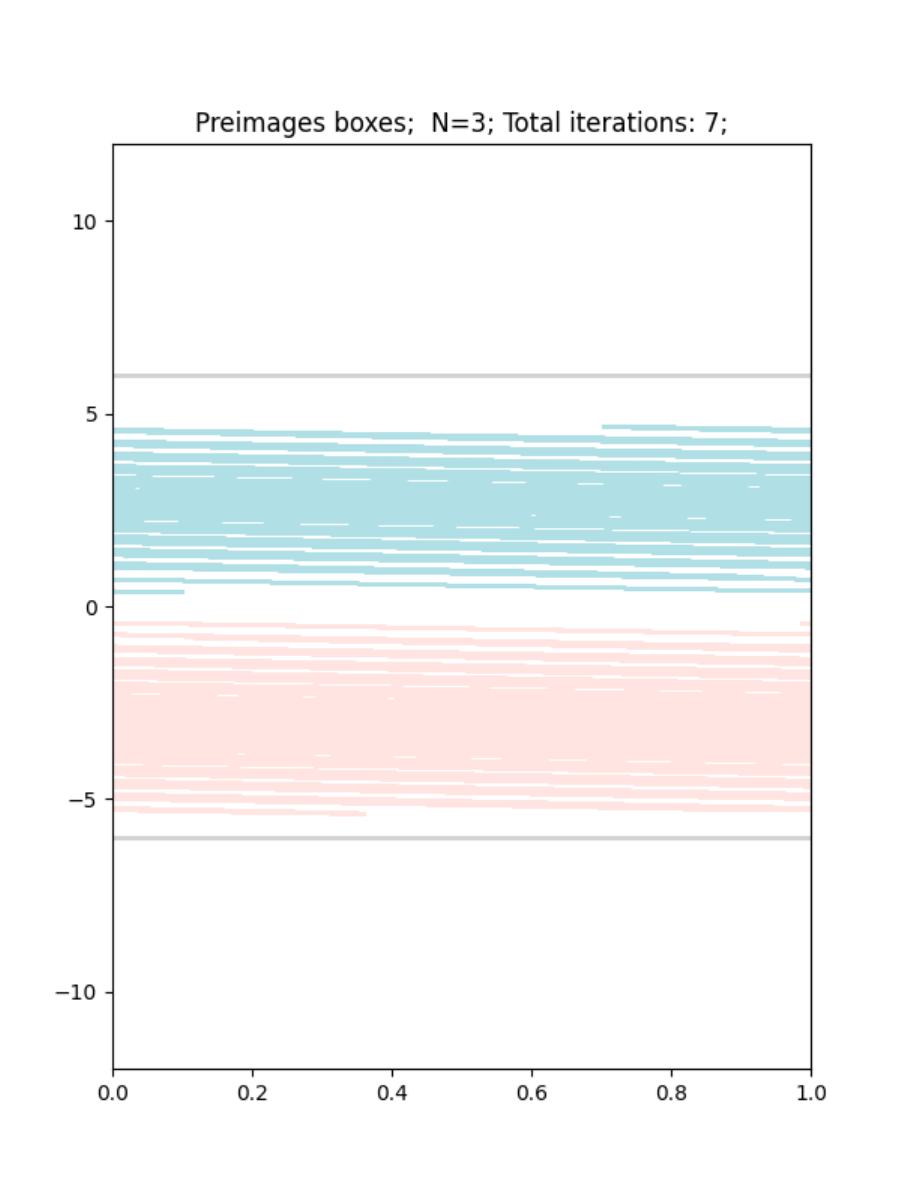}
\caption{}\label{dis2}
\end{minipage}
\end{figure}
\begin{figure}[h]
\begin{minipage}{0.35\linewidth}
\includegraphics[width=1\linewidth]{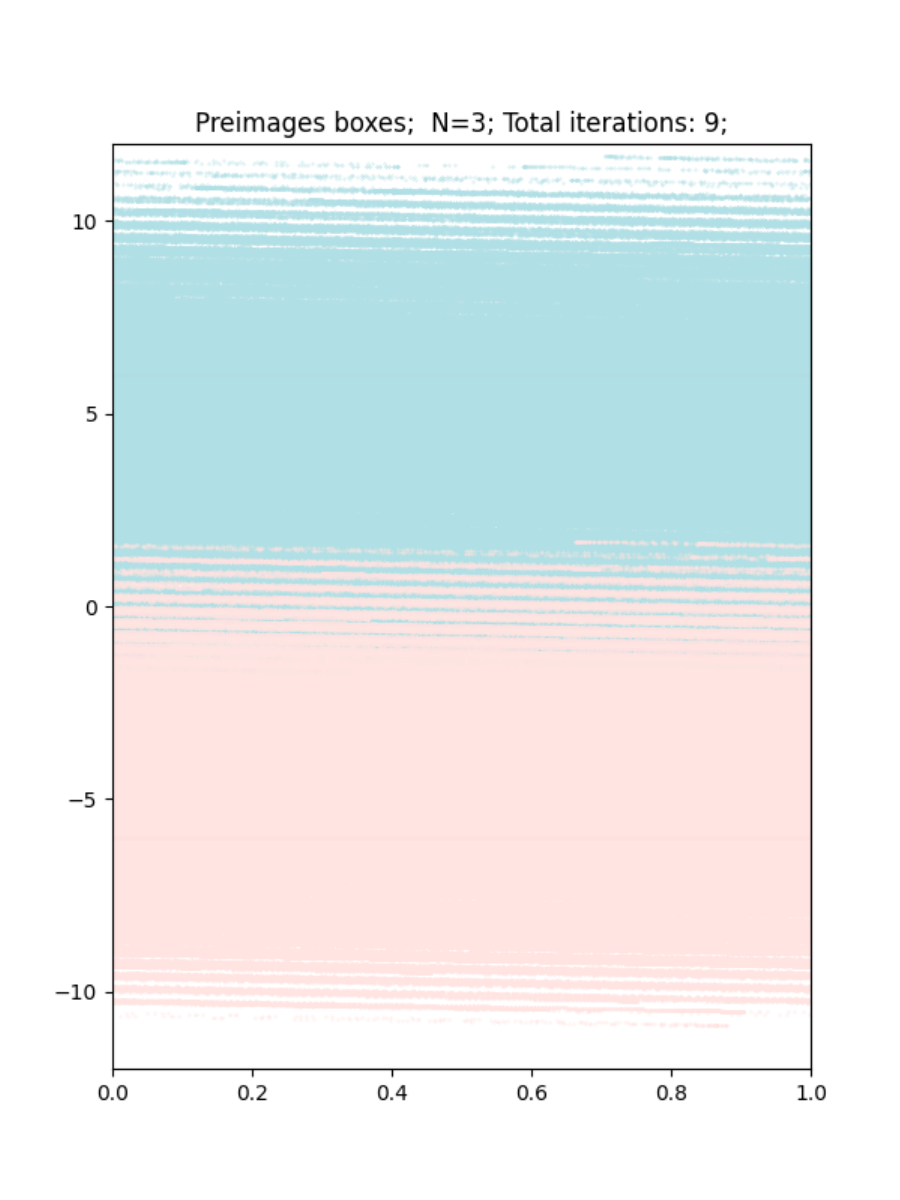}
\caption{}\label{dis3}
\end{minipage}
\begin{minipage}{0.35\linewidth}
\includegraphics[width=1\linewidth]{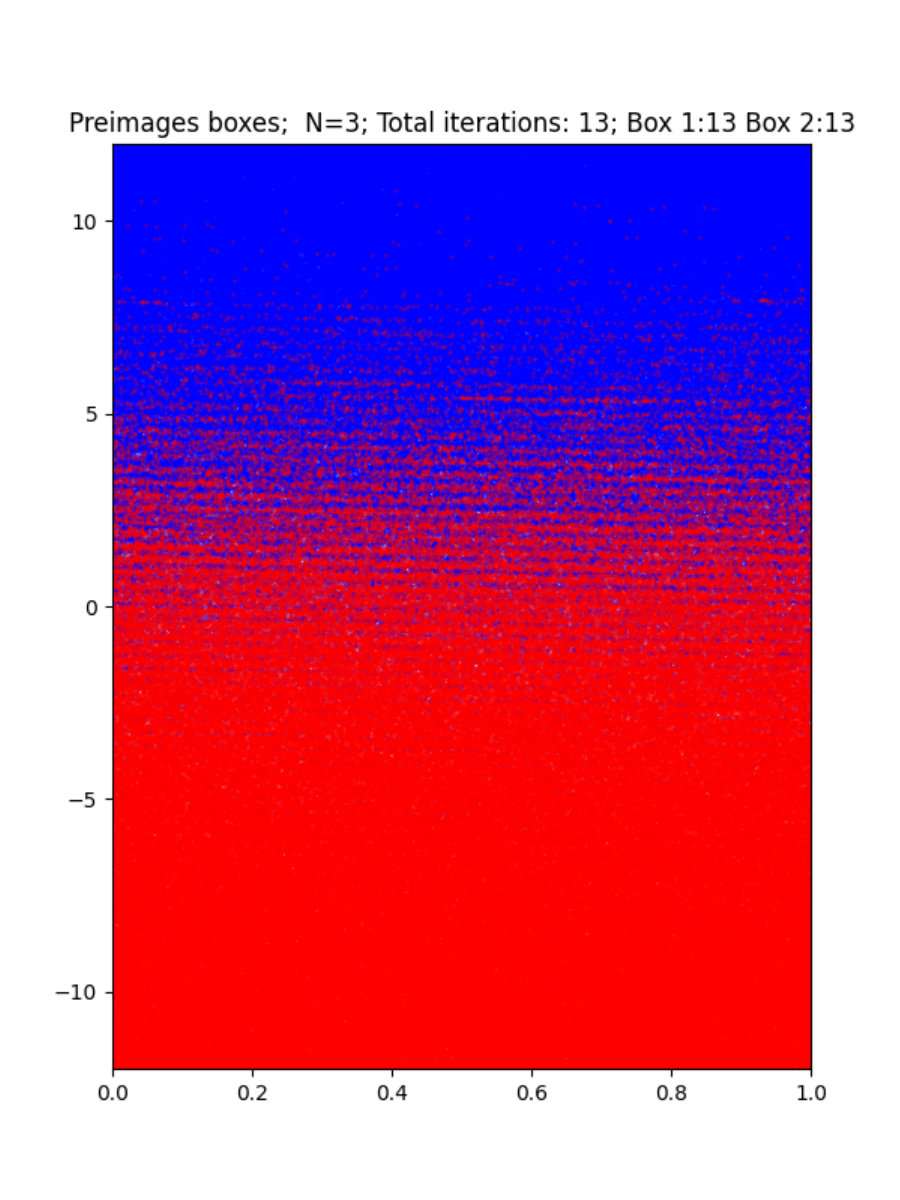}
\caption{}\label{dis4}
\end{minipage}
\end{figure}

Such visualization shows that for some pair of small neighborhoods of the fixed points $x_0,x_1$ we reach (after 13 steps under negative iterations) the lines given by the projections of $y=6$ and $y=-6$ as required in (1).
Having this in mind, we work in $C^{++}$ with the Library 
CAPD::DynSys in order to validate the existence
of the points $x_0^+,x_0^-,x_1^+,x_1^-$ (which is based on the Interval Arithmetic method), fulfilling the point (1) described above.

\smallskip

For point (2) %, which is necessary to obtain the CAP we are looking for, bearing in mind that we are working with the standard dissipative map,
we proceed as follows. Consider the four segments $S^+_0,S^-_0,S^+_1,S^-_1,$ where for $i\in\{0,1\}$, $S^{\pm}_i$ has endpoints $x_i$ and  $x_i^{\pm}$, obtained from the numerical validation. For each of these segments
we want to show that their images by $f_{a,b}^k$ with $k=0,1,\ldots,N$ is contained in the interior of the box $B_i$ centered at $x_i$, with width 1 and height 2, given that in such a case,
$S_0^+\cup S_0^-$ and $S_1^+\cup S_1^-$ can be extended to an $N$-disjoint pair of neighborhoods and we can
apply Theorem C.
There are two natural ways for checking this property. The most naive approach is to consider the maximum expansion $\eta_{a,b}$ associated to $f_{a,b}$, which is given by the supremum norm of the Jacobian $J_{x,y}f_{a,b}$, and then check the inequality $|S^j_i| \,\eta_{a,b}^N\leq \frac{1}{2}$. Unfortunately, this rather crude analytical approach
fails in some of the cases where we have obtained the CAPs.  The second way is to fix some small $\varepsilon$,
and consider an array of points $P=\{z_1,\ldots ,z_n \}$ inside the segment $S^j_i$, such that the gaps
between these points have a maximum length $\delta$ verifying $\delta\,\eta_{a,b}^N<\varepsilon$.
Then, by considering the new box $B_i^{\varepsilon}$ obtained by removing those points in $B_i$ whose distance to the boundary is less than  $\varepsilon$, in case we obtain a numerical validation for 
$$f_{a,b}^j(z_k)\in B_i^{\varepsilon}\mbox{ for }i=0,1,\,j=0,1,\dots,N,$$
it is obtained a CAP of point (2) for the continua $U_0=S_0^+\cup S_0^-$ and $U_1=S_1^+\cup S_1^-$. And this is the
approach that worked for our successful implementations.

\smallskip

The next table shows a list of parameters for the  Dissipative Standard Family for which we have obtained a CAP of the existence of chaos by the method described above. The work in preparation \cite{theoappandnumer} contains the details for the implemented software (which are more sophisticated versions than the basic ones used here), and explores
a much wider range of parameters and other families of maps (such
as non-twist Dissipative Standard Maps and the two-dimensional Arnold Family). %However, we mention here that for those cases introduced in this article, the used softwares are completely straightforward to be build. 

\begin{table}[ht]
\caption{\textbf{D.S.F.}} % title of Table
\centering % used for centering table
\begin{tabular}{|c |c | c|} % centered columns (4 columns)
\hline %inserts double horizontal lines
\ \ \textbf{b}\ \  &\ \ \textbf{a}\ \  & \ \  \textbf{Max. needed back. Iter.}  \\ [0.5ex] % inserts table
%heading
\hline % inserts single horizontal line
0.8 & 3  & 12\\
0.7 & 3  & 10\\
0.6 & 3  & 10\\
0.5 & 3  & 10\\
0.4 & 3  & 10\\
0.3 & 3  & 9\\
0.2 & 3  & 9\\
\hline %inserts single line
\end{tabular}
\label{table:nonlin} % is used to refer this table in the text
\end{table}

One important aspect here is given by the fact that the method can be applied for a wide range of Jacobian values
of the map, given by $b$. One should compare this to the known literature, where usually results proving positive entropy are known only for very small Jacobian values, see
\cite{linyoungkicked}. When the parameter $b$ is close to zero, one needs to increase $a$ so as to find the fixed points required by the theoretical results. For a study of small values of $a$ one should work with powers of the original map so as to ensure the existence of the necessary pairs of fixed points. In this first presentation we work with a fixed parameter $a=3$
for the sake of simplicity, whereas in the forthcoming work we show applications for a
wide range of parameters $a$.

\bibliographystyle{plain}
\bibliography{bibliografia2}

\end{document}